\documentclass{amsart}
\usepackage{upgreek} 
\usepackage{lmodern}
\usepackage{anyfontsize}
\usepackage{amssymb, mathtools, thmtools, amscd, xfrac, xspace}
\usepackage[dvipsnames]{xcolor}
\usepackage{amsmath}
\usepackage[alphabetic, msc-links]{amsrefs}

\usepackage{cleveref}
\hypersetup{
	colorlinks=true,
	linkcolor=RubineRed,
	citecolor=MidnightBlue,
	filecolor=Magenta,      
	urlcolor=cyan,
}
\usepackage{enumerate}
\usepackage{stmaryrd} 
\usepackage[mathcal]{eucal}
\usepackage{verbatim}
\usepackage{array,float}
\usepackage{longtable}
\usepackage{multirow}
\usepackage[margin=1.1in]{geometry}
\usepackage{xy}
\input xy
\xyoption{all}
\usepackage{pdflscape} 
\usepackage{caption}

\numberwithin{equation}{section}
\newtheorem{thm}{Theorem}[section]
\newtheorem{proposition}[thm]{Proposition}

\newtheorem{corollary}[thm]{Corollary}
\newtheorem{lemma}[thm]{Lemma}

\theoremstyle{definition}
\newtheorem{definition}[thm]{Definition}
\newtheorem*{remark*}{Remark}

\newcommand{\GL}{\mathbf{GL}}
\newcommand{\B}{\mathbf{B}}

\newcommand{\co}{\mathfrak{o}}
\newcommand{\mat}[4]{\begin{psmallmatrix}
		#1 & #2 \\  #3 & #4 \end{psmallmatrix} }

\newcommand{\Z}{\mathbf Z}

\newcommand{\C}{\mathrm{C}}
\newcommand{\s}{\mathbf{S}}

\renewcommand{\H}{\mathbf{H}}
\renewcommand{\det}{{\mathrm{det}}}

\newcommand{\G}{\mathbf{G}}

\renewcommand{\SS}{\mathbf{ss}}
\renewcommand{\ss}{\mathbf{ss}}

\newcommand{\IR}{\mathbf{cus}}

\newcommand{\NS}{\mathbf{sns}}
\newcommand{\sns}{\mathbf{sns}}

\newcommand{\lup}{{\ell_2}}
\newcommand{\ldown}{\ell_1}
\newcommand{\ld}{{\ell}_1}
\newcommand{\lu}{{\ell}_2} 
\newcommand{\cO}{{\mathfrak{o}}}
\newcommand{\ind}{\mathrm{Ind}}
\newcommand{\Ol}{{\cO_\ell}}
\newcommand{\vtc}{{V^t_\chi}}
\newcommand{\wtc}{{W^t_{\tilde{\chi}}}}
\newcommand{\K}{\mathbf{K}}
\newcommand{\g}{ \mathfrak{gl}_2}
\newcommand{\M}{{\mathfrak{gl}}}
\newcommand{\tC}{\mathrm{C}}

\newcommand{\rmf}{\mathrm{f}}
\newcommand{\U}{\mathbf{U}}
\newcommand{\uol}{{\U_{\ell}}}
\newcommand{\gol}{{\G_{\ell}}}
\newcommand{\bol}{{\B_{\ell}}}
\newcommand{\zol}{{\Z_{\ell}}}
\newcommand{\ztol}{{\Z^t_{\ell}}}

\newcommand{\I}{\mathrm{I}}
\newcommand{\hta}{{H^t}}
\newcommand{\OO}{\cO}
 
\newcommand{\glol}{{\GL_2(\cO_\ell)}}
\newcommand{\blol}{{\mathbf{B}(\cO_\ell)}}
\newcommand{\ulol}{{\mathbf{U}(\cO_\ell)}}
\newcommand{\zlol}{{\mathbf{Z}(\cO_\ell)}}
\newcommand{\T}{\mathbf{T}}
\newcommand{\tol}{{\T_\ell}}
\title[Gelfand pairs and Gelfand-Graev modules of $\GL_2(\cO_\ell)$]{On Gelfand pairs and degenerate Gelfand-Graev modules of General Linear groups of degree two over principal ideal local  rings of finite length}

\author{Archita Gupta}

\author{Pooja Singla}
\address{Department of Mathematics and Statistics, IIT Kanpur,  Kanpur 208016, India}
\email{architagup20@iitk.ac.in, psingla@iitk.ac.in}



\begin{document}
	\begin{abstract}	
Let $R$ be a principal ideal local ring of finite length with a finite residue field of odd characteristic. Denote by $G(R)$ the general linear group of degree two over $R$, and by $B(R)$ the Borel subgroup of $G(R)$ consisting of upper triangular matrices. In this article, we prove that the pair $(G(R), B(R))$ is a strong Gelfand pair.
We also investigate the decomposition of the degenerate Gelfand-Graev (DGG) modules of 
$G(R)$. It is known that the non-degenerate Gelfand Graev module (also called non-degenerate Whittaker model) of $G(R)$ is multiplicity-free. We characterize the DGG-modules where the multiplicities are independent of the cardinality of the residue field. We provide a complete decomposition of all DGG modules of $G(R)$ for $R$ of length at most four. 
	\end{abstract}
	\maketitle
\section{Introduction}
\noindent Let $\mathbf F$ be a non-archimedean local field, and let $\OO$ denote its ring of integers. The residue field $\mathbf k$ is finite of characteristic $p \geq 3$ and $| \mathbf{k} | =q.$
Let $\mathfrak{p}$ be the unique maximal ideal of $\OO$, and fix a generator $\varpi$ of $\mathfrak{p}.$
Consider $\G$ as a split reductive group scheme defined over $\cO$, and let $\G(\cO)$ represent the set of $\OO$-points of $\G.$ Being maximal compact subgroups, $\G(\cO)$ play a significant role in the representation theory of the groups $\G(\mathbf F).$ For $\ell \geq 1,$ let $\OO_{\ell} := \OO/ \mathfrak{p}^\ell.$ Since the groups $\G(\cO)$ are pro-finite, every complex continuous finite dimensional irreducible representation of $\G(\cO)$ factors through the quotient $\G(\cO_\ell)$ for some $\ell.$ Therefore, it suffices to study the finite-dimensional complex representations of the groups $\G(\cO_\ell)$. 
The study of irreducible representations of $\G(\cO_\ell)$ and their properties is a highly active area of research in mathematics, see \cites{MR4719887, MR4704476, MR4362775,MR4308643, MR3995719, MR3666059, MR3011874, MR2684153} and the references mentioned therein. In this article, we explore a few problems regarding the Gelfand pairs and Gelfand modules of $\G(\cO)$.

Recall that a pair $(G,H)$ is called a Gelfand pair if the induced representation $\ind_H^G(\mathrm{triv}_H)\cong\mathbb{C}[G/H]$ is multiplicity-free. In the study of finite groups, Gelfand pairs are valuable tools used across various branches of mathematics, including number theory (see \cite{MR1074028}), combinatorics (see \cites{MR2869097,MR3971529} ) and Harmonic analysis (see \cites{MR2816419,MR2389056}). 
This concept has been further generalized, leading to the notion of strong Gelfand pairs. A pair 
$(G,H)$ is defined as a strong Gelfand pair if
$\ind_H^G(\theta)$ is multiplicity-free for {\bf every} irreducible representation $\theta$ of $H$. In case of summetric group $S_n$, the pair $(S_n, S_{n-1})$ is known to be a strong Gelfand pair and this property provides a basis of the irreducible representations of $S_n.$ 

Let $\glol$ be the general linear group with entries from the finite ring $\cO_\ell$ and let $\blol$ be the Borel subgroup of $\glol$ consisting of the upper triangular matrices. It is well known that the pair $(\GL_2(\cO_1), \B(\cO_1))$ is a strong Gelfand pair. We extend this result to all $\cO_\ell$ and prove the following: 

\begin{thm}
\label{thm:borel-strong-Gelfand-pair}
    The pair $(\GL_2(\cO_\ell),\blol)$ is a strong Gelfand pair for every $\ell \geq 1$. 
\end{thm}

Let $\mathbf{P}_n(\Ol)$ be the mirabolic subgroup of $\GL_n(\Ol)$, that is, $\mathbf{P}_n(\Ol)$ denote the subgroup of $\GL_n(\Ol)$ consisting of matrices with the last row $(0,0,\cdots, 0, 1)$. 
For $\ell = 1$, it is well known that the pair $(\GL_n(\cO_1),\mathbf{P}_n(\cO_1))$ is a strong Gelfand pair, see \cite{MR643482}*{Section~13}. Let $\zlol$ be the center of $\glol$. Then  $\blol = \zlol \mathbf{P}_2(\cO_\ell) $. The following is an immediate corollary of \autoref{thm:borel-strong-Gelfand-pair}.  
\begin{corollary}
\label{cor:mira-induction}
 For any irreducible representation $\theta$ of $\mathbf{P}_2(\cO_\ell) $, the representation $\mathrm{Ind}_{\mathbf{P}_2(\cO_\ell) }^{\glol}(\theta)$ is a  multiplicity-free representation of $\glol$.   
\end{corollary}
This result is known to be true for $\dim(\theta) = 1$ and   $\dim(\theta) = (q-1)q^{\ell-1}$ by \cite{MR0338274}*{Proposition~1} and \cite{MR4399251}*{Theorem~1.1}, respectively. A proof of \autoref{thm:borel-strong-Gelfand-pair} is included in \autoref{sec:proof-borel-strong-Gelfand-pair}. To prove \autoref{thm:borel-strong-Gelfand-pair}, we first consider a generalization of the well-known Gelfand-Graev (GG) modules (also called  non-degenerate Whittaker models) of $\glol$. These GG-modules of $\GL_n(\cO_\ell)$ were studied in \cites{MR1334228, MR4356279, MR4399251}, and we briefly recall these here for $\glol$. Let $\ulol$ be the group of upper triangular unipotent matrices in $\glol.$ A non-trivial one dimensional representation $\psi\colon\Ol \rightarrow \mathbb{C}^{\times}$ 
such that $\psi|_{\varpi^{\ell -1} \Ol} \neq 1$ is called a primitive character of $\cO_\ell.$ For any primitive character $\psi$, define a one dimensional representation $\psi_t : \ulol\rightarrow \mathbb{C}^{\times}$ by
\begin{equation*}
\psi_t\begin{pmatrix} 1 & u \\ 0 & 1   \end{pmatrix} := \psi(\varpi^{\ell-t} u).
\end{equation*}
Any one dimensional representation of $\ulol$, upto conjugation by $\blol$, is of the form 
$\psi_t$ for some $t \in [0, \ell]$. The character $\psi_t$ for $t = \ell$ is called a non-degenerate character of $\ulol$ and 
$\glol$-representation space $V^\ell := \mathrm{Ind}_{\ulol}^{\glol}(\psi_\ell)$ is called the Gelfand-Graev (GG) module (or non-degenerate Whittaker model) of $\glol$. 
Generalizing this construction, 
$\psi_t$ for $t < \ell$  is called a {\it degenerate character} of $\ulol$ and $\glol$-representation space $V^t := \mathrm{Ind}_{\ulol}^{\glol}(\psi_t)$ for $t < \ell$ is called a degenerate Gelfand-Graev (DGG) module of $\glol$. This is a natural generalization of the well known parallel notions from $\GL_n(\mathbb F_q)$ \cite{MR643482}*{Section~12}. By \cite{MR4399251}, the GG module of $\GL_2(\cO_\ell)$ is known to be multiplicity free. Moreover, it is proved that an irreducible representation of $\GL_2(\cO_\ell)$ is a constituent of the GG module if and only if it is a regular representation of $\glol$, see \autoref{sec:preliminaries} for related definitions. 

The induced representation $\mathrm{Ind}_{\blol}^{\glol}(\theta)$, where $\theta$ is an irreducible representation of $\blol$, is a sub-representation of certain DGG-modules of $\glol.$ Hence, to prove ~\autoref{thm:borel-strong-Gelfand-pair},  we also consider the problem of decomposition of DGG-modules of $\glol$. In particular, we are interested in the multiplicities of their irreducible constituents. For $\ell = 1$, the corresponding decomposition is well known, see \cite{MR0696772}. We will consider $\ell \geq 2$ in this article. To describe our results, we define the following:
\[
a(t, \ell) = \underset{  \rho \in \mathrm{Irr}(\glol)  }{\mathrm{Sup}} \{ m_\rho \mid  \,\, \mathrm{dim}_\mathrm{\mathbb C}(\mathrm{Hom}_{\glol}(V^t, \rho)) = m_\rho \}. 
\]
The number $a(t, \ell)$ gives us the highest possible multiplicity of the irreducible constituents of the DGG-module $V^t$ and will be called its multiplicity bound. We prove the following result regarding $a(t, \ell).$ 
\begin{thm}
\label{thm:multiplicity-bounds}
     The multiplicity bound $a(t, \ell)$ of $V^t$ is independent of the residue field if and only if $t \geq  \ell-1.$
\end{thm}
As mentioned earlier, we have $a(\ell, \ell) = 1$ for all $\ell$. We remark that in view of the above results with $a(t, \ell)$ dependent on the residue field, \autoref{thm:borel-strong-Gelfand-pair} is a somewhat surprising result. 
See \autoref{sec:proof-multiplicity-bounds} for a proof of \autoref{thm:multiplicity-bounds}. In the course of proving the above result, we also obtain the following: 
\begin{thm} 
\label{thm:intermediate-steps-multiplicity-bounds}
\begin{enumerate}
\item For $\ell \geq 1,$ we have $a({\ell-1}, \ell) = 2.$  
\item For any cuspidal representation $\rho$ of $\glol$, we have 
\[
\mathrm{dim}_\mathrm{\mathbb C}(\mathrm{Hom}_{\glol}(V^t, \rho)) = \begin{cases} 
1, & \mathrm{for} \,\, t = \ell;\\ 
0, & \mathrm{for} \,\, t < \ell. 
\end{cases}
\]
\item For any split semisimple  representation $\rho$ of $\glol$, we have $\mathrm{dim}_\mathrm{\mathbb C}(\mathrm{Hom}_{\glol}(V^t, \rho)) = 2.$
\end{enumerate}
\end{thm} 
Hence, if the multiplicity bound of $V^t$ satisfies $a(t, \ell) > 2,$ then there exists a split non-semisimple irreducible representation $\rho$ of $\glol$ such that $\mathrm{Hom}_{\glol}(V^t, \rho) = a(t, \ell).$ We remark that we have not bee able to decompose $V^t$ completely into its irreducible constituents. 
At last, we include a complete decomposition of the DGG-modules of  $\GL_2(\cO_2)$, $\GL_2(\cO_3)$ and $\GL_2(\cO_4)$ for $t \leq \ell-1$ into their irreducible constituents in \autoref{sec:examples}.  
This, in particular, gives the structure of the endomorphism algebra of $V^t.$ Note that  $V^t \cong \oplus_{\chi \in \widehat{\zlol}} \vtc$, where $\vtc =  \ind_{\zlol \ulol}^{\glol} (\chi \otimes \psi_t).$ 
The structure of the endomorphism algebra of $V_t$ is then obtained from  $\mathrm{End}_{\glol}(V^t) \cong \prod_{\chi \in \widehat{\zlol}} \mathrm{End}_{\glol}(\vtc).$ Therefore, it suffices to understand $\mathrm{End}_{\glol}(\vtc)$ for $\chi \in \widehat{\zlol}$.

\begin{thm}
\label{thm:structure-endomorphism-algebra}
Let $\chi \in \widehat{\zlol}$ and $\bar{\chi}$ be a character of $\Z(\cO_{\ell-1})$ such that $\bar{\chi} = \chi \circ \chi_0^2$ for some $\chi_0 \in \widehat{\zlol}$. For $\ell \leq 4$, the endomorphism algebras $\mathrm{End}_{\glol}(\vtc)$  are given as follows:
    \begin{enumerate} 
      \item For $\ell = 1$, 
    \begin{itemize}
        \item[(a)]   
        $\mathrm{End}_{\GL_2(\co_1)}(V^0_\chi) \cong 
        \begin{cases}
          {(M_2(\mathbb{C}))}^{\oplus \frac{1}{2}{(q-3)} } \oplus {\mathbb{C}}^{\oplus 4} & \chi \in (\mathbb{F}_q^\times)^2;\\
            {(M_2(\mathbb{C}))}^{\oplus \frac{1}{2}{(q-1)}} &  \chi \in \mathbb{F}_q \setminus (\mathbb{F}_q^\times)^2. 
        \end{cases}$
        
         \item[(b)] 
          $\mathrm{End}_{\GL_2(\co_1)}(V^1_\chi) \cong {\mathbb{C}}^{\oplus q(q-1)}.$
        \end{itemize}
    \item For $\ell = 2$, 
    \begin{itemize}
        \item[(a)] $\mathrm{End}_{\GL_2(\co_2)}(V^0_\chi) \cong \mathrm{End}_{\GL_2(\co_1)}(V^0_{\bar{\chi}})\oplus {(M_2(\mathbb{C}))}^{\oplus \frac{1}{2}{(q-1)}^2 }\oplus {(M_{q-1}(\mathbb{C}))}. $

        \item[(b)] $\mathrm{End}_{\GL_2(\co_2)}(V^1_\chi) \cong \mathrm{End}_{\GL_2(\co_1)}(V^1_{\bar{\chi}})\oplus {(M_2(\mathbb{C}))}^{\oplus \frac{1}{2}{q(q-1)}^2 } \oplus {(\mathbb{C})}^{q-1}.$

        \item[(c)] $\mathrm{End}_{\GL_2(\co_2)}(V^2_\chi) \cong {\mathbb{C}}^{\oplus q^3(q-1)}.$

    \end{itemize}
    
    \item  For $\ell = 3$,
     \begin{itemize}
    \item[(a)] $\mathrm{End}_{\GL_2(\co_3)}(V^0_\chi) \cong  \mathrm{End}_{\GL_2(\co_2)}(V^0_{\bar{\chi}})\oplus {(M_2(\mathbb{C}))}^{\oplus \frac{1}{2}{q{(q-1)}^2} }\oplus {(M_{q-1}(\mathbb{C}))} \oplus {(M_{2(q-1)}(\mathbb{C}))}^{\oplus \frac{q-1}{2}}.$
    
    \item[(b)] $\mathrm{End}_{\GL_2(\co_3)}(V^1_\chi) \cong 
 \mathrm{End}_{\GL_2(\co_2)}(V^1_{\bar{\chi}})\oplus {(M_2(\mathbb{C}))}^{\frac{1}{2}{q(q-1)}^2}\oplus{(M_{q-2}(\mathbb{C}))}^{\oplus \frac{q-1}{2} } \oplus M_{q-1}(\mathbb{C})\oplus \newline {(M_{q}(\mathbb{C}))}^{\oplus \frac{q-1}{2} }.$
     
    \item[(c)] $\mathrm{End}_{\GL_2(\co_3)}(V^2_\chi) \cong \mathrm{End}_{\GL_2(\co_2)}(V^2_{\bar{\chi}})\oplus {(M_2(\mathbb{C}))}^{\oplus \frac{1}{2}{q(q-1)}^2 } \oplus {(\mathbb{C})}^{q(q-1)}.$

    \item[(d)] $\mathrm{End}_{\GL_2(\co_3)}(V^3_\chi) \cong {\mathbb{C}}^{\oplus q^5(q-1)}.$

    \end{itemize}
    \item For $\ell = 4$,
    \begin{itemize}
        \item[(a)] $\mathrm{End}_{\GL_2(\co_4)}(V^0_\chi) \cong 
        \mathrm{End}_{\GL_2(\co_3)}(V^0_{\bar{\chi}})\oplus {(M_2(\mathbb{C}))}^{\oplus \frac{1}{2}{q^2{(q-1)}}^2 } \oplus {(M_{2(q-1)}(\mathbb{C}))}^{\oplus {\frac{q-1}{2}}q}\oplus {(M_{q^2-q}(\mathbb{C}))}.$
        \item[(b)] $\mathrm{End}_{\GL_2(\co_4)}(V^1_\chi) \cong \mathrm{End}_{\GL_2(\co_3)}(V^1_{\bar{\chi}}) \oplus {(M_2(\mathbb{C}))}^{\frac{1}{2}{q^2(q-1)}^2}\oplus {(M_{q}(\mathbb{C}))}^{\oplus q-1} \oplus {(M_{2(q-1)}(\mathbb{C}))}^{\oplus \frac{q-1}{2}q}.$
    
      \item[(c)] $\mathrm{End}_{\GL_2(\co_4)}(V^2_\chi) \cong  
 \mathrm{End}_{\GL_2(\co_3)}(V^2_{\bar{\chi}}) \oplus  {(M_2(\mathbb{C}))}^{\frac{1}{2}{q^2(q-1)}^2}\oplus {(M_{q-1}(\mathbb{C}))}^{\oplus q} \oplus {(M_{q-2}(\mathbb{C}))}^{\oplus \frac{q-1}{2}q} \oplus {(M_{q}(\mathbb{C}))}^{\oplus \frac{q-1}{2}q}.$
    
    \item[(d)] $\mathrm{End}_{\GL_2(\co_4)}(V^3_\chi) \cong \mathrm{End}_{\GL_2(\co_3)}(V^3_{\bar{\chi}}) \oplus {(M_2(\mathbb{C}))}^{\oplus \frac{1}{2}{(q-1)}^2 q^2} \oplus {(\mathbb{C})}^{\oplus {q^2(q-1)}}.$

    \item[(e)] $\mathrm{End}_{\GL_2(\co_4)}(V^4_\chi) \cong {\mathbb{C}}^{\oplus q^7(q-1)}.$
    \end{itemize}
    \end{enumerate} 
\end{thm}

 See \autoref{sec:examples} for a proof of the above result. The following result is obtained directly from \autoref{thm:structure-endomorphism-algebra}. 
\begin{corollary}
\label{cor:multiplicity-bounds-examples}
For $(t,\ell)$ such $(0,2) \leq (t, \ell) \leq (4,4)$ and $t < \ell-1$, the multiplicity bounds $a(t, \ell)$ are given as follows:   
\begin{center}
    \begin{tabular}
    {|c||c||c|c||c|c|c|c|c|c|}
    \hline
    ${(t,\ell)}$ & $(0,2)$ & $(1,2)$ & $(0,3)$ & $(1,3)$ & $(2,3)$ & $(0,4)$ & $(1,4)$ & $(2,4)$ & $(3,4)$\\
    \hline 
    $a(t, \ell)$ & $(q-1)$ & $2$ & $2(q-1)$ & $q$ &  $q^2-q$ &  $2$ & $2(q-1)$ & $q$ & $2$\\
    \hline 
    \end{tabular}
    \end{center} 

\end{corollary}

 At last, we remark that the  questions regarding the decomposition of DGG modules of  $\GL_n(\cO_\ell)$ for $n \geq 3$  are largely open. For related results, see \cites{MR2504482, MR2597022, MR3119211}.

\section{Preliminaries}
\label{sec:preliminaries}
In this section, we set up the notation and briefly recall the construction of the irreducible representations of $\GL_2(\cO_\ell)$ from \cite{MR3737836}.
We will also discuss alternative constructions of some representations of $\GL_2(\cO_\ell)$. 

Throughout this article, we consider $\G = \GL_2$. Let $\B$
be the Borel subgroup of $\G$ consisting of all upper triangular matrices, and let $\U$ and $\T$ denote the unipotent radical and torus of $\B$, respectively. The center of $\G$ will be denoted by $\Z$ and the Lie algebra of $\G$ will be denoted by $\g$. For a group scheme $\H$ defined over $\cO$, we shall denote $\H(\co_\ell)$ by $\H_\ell$. In this article, we will always consider the complex representations of $\H_\ell$. For a group $G$, the set of all in-equivalent irreducible representations of $G$ will be denoted by $\mathrm{Irr}(G).$ For a subgroup $H$ of $G$ and a representation $\phi$ of $H$, the set $\{\rho \in \mathrm{Irr}(G) \mid \langle \rho|_{H}, \phi \rangle \neq 0 \}$ will be denoted by $\mathrm{Irr}(G \mid \phi).$ For any representation $\rho$ of $G$ such that $\rho \in \mathrm{Irr}(G \mid \phi),$ we say $\rho$ lies above $\phi$. In this article, by a character of a group $G$, we always mean a one-dimensional representation of $G$. For an abelian group $A$, the set $\mathrm{Irr}(A)$ is also denoted by $\widehat{A}$.

For $ i \leq \ell$ and the natural projection maps  $\rho_{i}\colon \gol  \rightarrow \mathrm{\G}_i,$  let $\K^{i} =  \ker (\rho_{i}) $ be the $i$-th congruence subgroups of $\gol.$  For $i \geq \ell/2,$  the group  $\K^i $ is isomorphic to the abelian additive subgroup $\M_2(\cO_{\ell-i}).$ Let $\psi \colon \cO_\ell  \rightarrow \mathbb {C}^\times$ be a fixed primitive one dimensional representation of $\cO_\ell$. 
 For any $i \leq \ell/2$ and  $x \in \M_2(\cO_i),$ let $\hat{x} \in \M_2(\cO_\ell)$ be an arbitrary lift of $x$ satisfying $\rho_{i}(\hat{x}) = x.$ Define $\psi_x: I + \varpi^{\ell-i} \M_2(\cO_\ell) \rightarrow \mathbb C^\times$ by 
 \[ 
 \psi_x(I + \varpi^{\ell-i} y) = \psi(\varpi^{\ell-i}\mathrm{tr}(\hat{x}y)),
 \]
for all $I + \varpi^{\ell-i} y \in \K^{\ell-i}.$  Then $\psi_x$ is a well defined one dimensional representation of $\K^{\ell-i} .$ Further, the following duality for abelian groups $\K^{i}$ and $\M_2(\cO_{\ell-i})$ holds for $i \geq \ell/2$. 
\begin{equation*}
\label{eq:duality}
\M_2(\cO_{\ell-i}) \cong \widehat {\K^{i}};\,\, x \mapsto \psi_x,\,\, \mathrm{where} \,\, \psi_x(y) = \psi(\varpi^{\ell-i}\mathrm{tr}(\hat{x}y))  \,\, \forall \,\, y \in \K^{i}. 
\end{equation*}

We say a one dimensional representation $\psi_x \in \widehat{\K^{i} }$ for $i \geq \ell/2$ is regular if and only if $x \in M_2(\cO_{\ell-i})$ is a regular matrix. It is well known that for $i \geq \ell/2$, the representation $\psi_x \in \widehat{\K^{i}}$ is regular if and only if $\psi_x|_{ \K^{\ell-1} }$ is regular.
We recall the following definition of types of regular matrices $x\in \g(\cO_\ell)$:
\begin{definition}
    A regular matrix $x\in \g(\cO_\ell)$ is called cuspidal (resp. split semisimple and split non-semisimple) if the characteristic polynomial of the image of $x$ in $\g({\mathbb{F}_q})$ under the natural projection map has no roots (resp. two distinct roots and a repeated root) in $\mathbb{F}_q$.
\end{definition}
An irreducible representation $\rho$ of $\gol$ is called regular (cuspidal, split semisimple, split non-semisimple) if the orbit of its restriction to $\K^{\ell-1} $ consists of one- dimensional representations $\psi_x$ for regular (cuspidal, split semisimple, split non-semisimple) $x$. 
We shall say a representation (matrix) is $\IR$, $\SS$, or $\NS$  if it is cuspidal, split semisimple, or split non-semisimple representation (matrix) respectively.
For $\gol,$ these are precisely the regular representations (matrices). Note that the regular representations (even upto twist) of $\gol$ can not be obtained  from $\G_{\ell-1}.$

 Denote $\lfloor \ell/2 \rfloor$ by $\ldown$ and  $\lceil \ell/2 \rceil$ by $\lup$. We now summarize very briefly the construction of regular representations of $\gol $ from \cite{MR3737836} with emphasis on the statements that we require in this article.

\subsection{Construction of regular representations of $\gol$ for $\ell $ even } 
\label{sec:E.construction}
Let $\psi_x \in \widehat{\K^{\ell/2}}$ be a regular one dimensional representation of $\K^{\ell/2}$ for $x \in \g(\cO_{\ell/2}).$ Then the following gives the construction in this case. Let $\s_x = \{ g \in \gol \mid \psi_x^g \cong \psi_x  \}  $ be the inertia group of $\psi_x$ in $\gol.$ 
 Then $\s_x = \C_{\gol}(\tilde{x})       \K^{\ell/2} ,$ where $\tilde{x} \in \cO_\ell$ is any lift of $x$ to $\g(\cO_\ell)$ and   $\C_{\gol}(\tilde{x})$ is the centralizer of $\tilde{x}$ in $\gol$.

  Let $\rho \in \mathrm{Irr} \left( \gol  \mid \psi_x     \right)$ be a regular representation of $\gol,$ then there exists an extension $\widetilde{\psi_x}$ of $\psi_x$ to $\s_x$ such that $\rho \cong \mathrm{Ind}_{\s_x}^{\gol} (\widetilde{\psi_x}) .$ 

\subsection{Construction of regular representations of $\gol$ for $\ell $ odd}  
\label{O.construction}
Let $\psi_x \in \widehat{\K^{\lup}}$ be a regular one dimensional representation of $\K^{\lup}$ for $x \in \g(\cO_{\lup}).$ Let $\tilde{x} \in \g(\cO_\ell)$ be a lift of $x$ to $\g(\cO_\ell).$ Define the group  $R_{\tilde{x}}=(\K^{\ldown} \cap \C_{\gol}(\tilde{x})) \K^{\lup}$.

The one dimensional representation $\psi_x $ extends to $R_{\tilde{x}}.$ Let $\widetilde{\psi_x}$ be an extension of $\psi_x$ to $R_{\tilde{x}}$ and $\sigma \in \mathrm{Irr}(\K^{\ldown} \mid \psi_x) $ be unique irreducible representation determined by $\widetilde{\psi_x}$ (using Heisenberg lift). Then, 
 \begin{equation}
 \label{eq: sigma restricted to radical}
\sigma|_{R_{\tilde{x}}} = \underbrace{\widetilde{\psi_x} + \cdots + \widetilde{\psi_x} }_{q-\mathrm{times}}. 
 \end{equation}

  Let $\s_{\sigma} = \{ g \in \G(\cO_\ell) \mid \sigma^g \cong \sigma  \}$ be the inertia group of $\sigma \in \mathrm{Irr}(\K^{\ldown} \mid \psi_x)$ in $\gol.$ 
  Then $\s_{\sigma}  =\s_x = \C_{\gol}(\tilde{x}) \K^{\ldown},$  where $\tilde{x} \in \g(\cO_\ell)$ is any lift of $x$ to $\g(\cO_\ell).$ 
  Every $\sigma \in \mathrm{Irr}(\K^{\ldown} \mid \psi_x)$ extends to the inertia group $\s_\sigma.$ 
  In particular, every such extension induces irreducibly to $\gol$ and gives rise to a regular representation of $\gol.$

Let $\mathrm{Irr}^{\SS}(\gol)$ ($\mathrm{Irr}^{\NS}(\gol)$, $\mathrm{Irr}^{\IR}(\gol)$) denote the set of all distinct $\SS$ ($\NS$, $\IR$)  irreducible representations of $\gol.$  The construction also gives that any two $\SS$ ($\NS$, $\IR$) representations of $\gol$ have the same dimension. Let $\dim_{\SS}$ ($\dim_{\NS}$,  
 $ \dim_{\IR})$  denote the dimensions of the corresponding representations of $\gol.$ The following table collects the well known information regarding these. 
\begin{table}[ht]
\renewcommand{\arraystretch}{1.6}
\centering
\begin{tabular}
{|c|c|c|}

\hline 
$\rho$ & \,\, $|\mathrm{Irr}^\rho(\gol)|$ \,\, &\,\, $\dim_\rho$ \,\, \\
\hline
$\SS$ & $\frac{1}{2}(q-1)^3 q^{2 \ell-3}$  & $(q+1)q^{\ell-1}$\\ 
\hline
 
$\NS$ & $(q-1)q^{2 \ell-2 }$  & $ (q^2-1)q^{\ell-2}$ \\
\hline
$\IR$ & $\frac{1}{2} (q-1)(q^2-1)q^{2 \ell-3}$ & $(q-1)q^{\ell-1}$ \\
\hline
\end{tabular}
\caption{Numbers and dimensions of regular representations of $\gol$}
\label{table:numbers-dimensions}
\centering
\end{table}
We also recall a natural extension of some characters of $\K^\lup$ to $\gol.$
Let $I_\alpha=\left(\begin{smallmatrix}
    \alpha & 0\\
    0 & \alpha
\end{smallmatrix}\right)\in  \M_2(\co_{\lup})$ be a scalar matrix. Let $\delta_{\alpha}$ be a character of the group $1+\varpi^{\lup}\co_\ell$ defined by
\[
\delta_{\alpha}(1+\varpi^{\lup}x)=\psi(\varpi^{\lup} \alpha x).
\]
Then $\delta_\alpha$ extends to the group $\co_\ell^\times$, call this extension $\tilde{\delta_{\alpha}}$. The character \begin{eqnarray}
\label{eq:natureal-extension}
\mu_\alpha=\tilde{\delta_{\alpha}}\circ \mathrm{det}
\end{eqnarray}
is a character of $\gol$ where $\mathrm{det}$ is the determinant map from $\gol$ to $\co_\ell^\times$. This character is an extension of $\psi_{I_\alpha}$ from $\K^{\lup}$ to $\gol$.
We now give a general result based on the above construction that we need in the later sections. Let  $A =\left(\begin{smallmatrix}   
 a  & b \\ 
\varpi^{\ell-t} & a+ \varpi^i d  
\end{smallmatrix}\right)\in \mathfrak{g}(\co_{\ldown})$ be a $\sns$ matrix. We denote $\zol \uol \cap \s_A$ by $\U_A$. 
\begin{lemma}
 \label{lem: odd case t=l-1}
     For $\ell$ odd and $\phi$ an extension of $\psi_A$ to $\U_A \K^\lup$, the representation $\ind_{\U_A \K^{\lup}}^{\U_A \K^{\ldown}}\phi$ is multiplicity free for $t \geq \lup$. 
 \end{lemma}
 
\begin{proof}
By the construction mentioned in \autoref{O.construction}, the representation  $\ind_{\U_A \K^{\lup}}^{\U_A \K^{\ldown}}\phi$ has dimension $q^2$ and its every irreducible constituent is $q$ dimensional. Consider the intertwiner 
  \[
     \langle \ind_{\U_A \K^{\lup}}^{\U_A \K^{\ldown}}\phi,\ind_{\U_A \K^{\lup}}^{\U_A \K^{\ldown}}\phi \rangle=\oplus_{g\in\U_A \K ^{\lup}\backslash \U_A \K^{\ldown}/ \U_A \K^{\lup}}{\langle \phi,\phi^g \rangle}_{{\U_A \K^{\lup}}\cap {(\U_A \K^{\lup})}^g}.
    \]
    By direct calculations, the set 
    $$\Upsilon=\{\I+ \varpi^{\ldown}\begin{pmatrix}
        w & 0\\
        z & -w
    \end{pmatrix}\mid  w,z \in \cO_\ell \smallsetminus \varpi (\co_\ell)\}$$
    is an exhaustive set of double cosets representatives $\U_A \K ^{\lup}\backslash \U_A \K^{\ldown}/ \U_A \K^{\lup}$.
Therefore, the above intertwiner is atleast $q$ (in case of multiplicity free) and at most $q^2$ (the number of coset representatives). For $t \geq \lup,$ assume
\[
k=\begin{cases}
    \ldown-(\ell-t), & \text{ if } i=\ldown; \\
    \ldown-\text{min} \{i,\ell-t\}, & \text{ if } i<\ldown.
\end{cases}
\]
Then we have $X = \left(\begin{smallmatrix} 1 & \varpi^{ k}x \\ 0 & 1 \end{smallmatrix}\right) \in \U_A \K^{\lup} \cap (\U_A 
\K^{\lup})^g $ for all $g \in \Upsilon.$ By comparing $\phi(X) $ and $\phi^g(X)$ for $g \in \Upsilon$, there exists exactly $q$ such $g$ satisfying $\phi = \phi^g$ on $\U_A \K^\lup \cap (\U_A \K^{\lup})^g.$ Hence $\langle\ind_{\U_A \K^{\lup}}^{\U_A \K^{\ldown}}\phi,\ind_{\U_A \K^{\lup}}^{\U_A \K^{\ldown}}\phi\rangle = q$,  and therefore $\ind_{\U_A \K^{\lup}}^{\U_A \K^{\ldown}}\phi$ is multiplicity free.  
 \end{proof}

\begin{proposition} 
\label{prop:multiplicity-free-SA-representation}
Let $\ell \geq 1$ be odd, $t \geq \lup$ and $A =\left(\begin{smallmatrix}   
 a  & b \\ 
\varpi^{\ell-t} & a+ \varpi^i d  
\end{smallmatrix}\right)\in \mathfrak{g}(\co_{\ldown})$. Let $\tilde{\sigma}$ be an irreducible representation of $\s_A$ that lies above $\psi_A$ and ${\psi_t}_{\mid {\uol \cap \s_A}}$. Then $\tilde{\sigma}_{\mid {\U_A}}$ is multiplicity free.   
\end{proposition}
\begin{proof}
Let $\phi$ be an extension of $\psi_A$ and ${\psi_t}_{\mid {\uol \cap \s_A}}$ to $\U_A \K^{\lup}.$ Note $\U_A \K^{\lup} \subseteq \s_A.$ Let $\tilde{\sigma}$ be an irreducible representation of $\s_A$ such that $\rho \cong \mathrm{Ind}_{\s_A}^\gol (\tilde{\sigma}).$ Denote $\tilde{\sigma}|_{\U_A \K^{\ldown}}$ by $\tilde{\sigma}'$. By the construction of $\tilde{\sigma}$, the representation $\tilde{\sigma}'$ is irreducible. Hence it is enough to prove that $\tilde{\sigma}'|_{\U_A}$ is multiplicity-free. By \autoref{lem: odd case t=l-1} and Frobenius reciprocity, $\tilde{\sigma}'|_{\U_A \K^{\lup}}$ is multiplcity-free. By construction of $\tilde{\sigma},$ we have $\tilde{\sigma} |_{\K^\lup} \cong \psi_A^{\oplus q}.$ Hence we must have $\tilde{\sigma}' |_{\U_A}$ is multiplicity free.          
\end{proof}

We now describe known alternative constructions of $\ss$  and $\sns$  representations of $\gol$ (see \cite{MR2588859}*{Proposition 3} and \cite{MR2584957}*{Section 3.3.3} respectively). We will use these constructions as required. 
\subsection{Alternative construction of $\ss$ representations of $\gol$} 
\label{sec:construction-SS}

Let $(\chi_1,\chi_2)\in \widehat{\Ol^\times} 
 \times \widehat{\Ol^\times}$. Define a character of $\bol \subset \gol$ as 
\[
(\chi_1,\chi_2)\begin{pmatrix}
  x & z\\ 0 & y
\end{pmatrix}=\chi_1(x)\chi_2(y).
\]

 \begin{definition}(Injective character of $\cO_\ell^\times$) A character $\chi$ of $\cO_\ell^\times$ is called injective if $\chi|_{1+ \varpi^{\ell-1}\cO_\ell} \neq 1.$ 
\end{definition}

\begin{lemma}  
\label{lem:split-semisimple from borel induction}
Let 
$
C=\{(\chi_1,\chi_2)\in \widehat{\co_\ell^\times} \times \widehat{\co_\ell^\times} \mid \chi_1\chi_2^{-1}(1+\varpi^{\ell-1}x)\neq 1\}. 
$ Then
\begin{enumerate} 
\item $|C|=q^{2\ell-3}(q-1)^3$.
\item Let $(\chi_1,\chi_2),(\chi_3,\chi_4) \in C$. Then \[\ind_{\bol}^{\gol}(\chi_1,\chi_2)\cong \ind_{\bol}^{\gol}(\chi_3,\chi_4) \, \mathrm{if \, and \, only \, if}\,
(\chi_1,\chi_2) \in \{(\chi_3,\chi_4),(\chi_4,\chi_3)\}.\]
\item If $\rho\in \mathrm{Irr}^{\mathrm{ss}}(\gol)$ then $\rho\cong\ind_\bol^\gol(\chi_1,\chi_2)$ for some $(\chi_1,\chi_2) \in C$.
\end{enumerate}
\end{lemma}
\begin{proof}
(1) This follows easily by considering $|\cO_\ell^\times|$ and $|1+ \varpi^{\ell-1}\cO_\ell|.$ 

(2) From \cite{MR2584957}*{Lemma~4.1} we have $\bol \backslash \gol /\bol =  \{\left(\begin{smallmatrix}
       1 & 0\\
       \varpi^j & 1
\end{smallmatrix}\right),\left(\begin{smallmatrix}
       0 & 1\\
       1 & 0
\end{smallmatrix}\right)\mid 1\leq j\leq \ell \}$. Suppose $(\chi_1,\chi_2)\neq (\chi_3,\chi_4)$ then by Mackey's restriction formula and Frobenius reciprocity, we have
\[
\langle \ind_{\bol}^{\gol}(\chi_1,\chi_2),\ind_{\bol}^{\gol}(\chi_3,\chi_4)\rangle= {\langle(\chi_1,\chi_2), (\chi_4,\chi_3) \rangle}_{\tol} \oplus_{g\in \{\left(\begin{smallmatrix}
       1 & 0\\
       \varpi^j & 1
\end{smallmatrix}\right), 1\leq j \leq \ell\}}{\langle(\chi_1,\chi_2), {(\chi_3,\chi_4)}^g \rangle}_{\bol\cap \bol^g}
\]
By direct calculations, we obtain $${\langle(\chi_1,\chi_2), {(\chi_3,\chi_4)}^g \rangle}_{\bol\cap \bol^g}=0; \,\,\mathrm{for} \,\,  g \in \{\left(\begin{smallmatrix}
       1 & 0\\
       \varpi^j & 1
\end{smallmatrix}\right) 
\mid \ 1\leq j \leq \ell-1\}.$$ Hence $ \ind_{\bol}^{\gol}(\chi_1,\chi_2)\cong\ind_{\bol}^{\gol}(\chi_3,\chi_4)$ if and only if $(\chi_1,\chi_2)=(\chi_4,\chi_3)$. 

(3) Assume that $\lambda_1, \lambda_2 \in \widehat{\cO_\ell^\times}$ such that  $\chi_i(1+\varpi^{\lup}x)=\psi(\varpi^{\lup}\lambda_i x)$ for $1 \leq i \leq 2$. Since $\chi_1\chi_2^{-1}$ is injective, we also have  $\lambda_1 - \lambda_2 \in \co_{\ell}^{\times}.$
By Mackey's restriction formula,  $\psi_A$ for $A=\left(\begin{smallmatrix}
    \lambda_1 & 1\\
    0 & \lambda_2
\end{smallmatrix}\right)$ is a constituent of $\mathrm{Res}_{\K^{\lup}}^{\gol}\ind_{\bol}^{\gol} (\chi_1,\chi_2)$. 

 Since $A$ is a $\ss$ matrix, we obtain that $\ind_{\bol}^{\gol}(\chi_1,\chi_2)$ has $\ss$ representations as constituents. By \cite{MR2584957}*{Lemma~4.1}, $\ind_{\bol}^{\gol}(\chi_1,\chi_2)$ is an irreducible representation. Hence  $\ind_{\bol}^{\gol}(\chi_1,\chi_2)\in \mathrm{Irr}^{\mathrm{ss}}(\gol)$. Using (1) and (2), we obtain $\frac{1}{2}p^{2\ell-3}(p-1)^3$ number of distinct $\ss$ representations of $\gol$ of the form $\ind_{\bol}^{\gol}(\chi_1,\chi_2)$. By comparing this to $|\mathrm{Irr}^{\mathrm{ss}}(\gol)|$ from \autoref{table:numbers-dimensions}, we obtain our result. 
\end{proof}

  \subsection{Alternative construction of $\sns$ representations of $\gol$ for odd $\ell$}
  \label{sec: construction-SNS}
  In this section, we recall an alternative construction of $\sns$  representations of $\GL_2(\co_\ell)$ from \cite{MR2584957}*{Section 3.3.3}  for odd $\ell$.

  Let $A=\left(\begin{smallmatrix}
      \alpha & 1\\
      \varpi^j\beta & \alpha 
  \end{smallmatrix}\right) \in \M_2(\cO_{\ell})$ and consider the corresponding character $\psi_A$ of $\K^{\lup}$. Then $\s_A = \C_{\gol}(A)   \K^{\ldown} $ is given by 
  \begin{equation*}
      \s_A =\left\{\begin{pmatrix}
          x & z\\
          \varpi^j \beta z+\varpi^{\ldown}y & x+\varpi^{\ldown}w
      \end{pmatrix} \mid a,b,c,d \in \cO_\ell\right\}.
  \end{equation*}
  Consider a subgroup $\mathrm{N}=\left \{\left(\begin{smallmatrix}
      1+\varpi^{\ldown}x & \varpi^{\lup-j}z\\
       \varpi^{\lup}y    & 1+\varpi^{\ldown}w
             \end{smallmatrix}\right) \mid x,y,z,w \in \cO_\ell \right \}$ of $\gol$. It is a normal subgroup of $\s_A$ and we can extend $\psi_A$ to $\mathrm{N}$ as follows:
\begin{equation*}
    \psi_A'\begin{pmatrix}
      1+\varpi^{\ldown}x & \varpi^{\lup-j}z\\
       \varpi^{\lup}y   & 1+\varpi^{\ldown}w
             \end{pmatrix}=\psi(\varpi^{\lup} y+\varpi^{\lup}\beta z) \mu_\alpha\begin{pmatrix}
      1+\varpi^{\ldown}x & \varpi^{\lup-j}z\\
       \varpi^{\lup}y    & 1+\varpi^{\ldown}w
             \end{pmatrix},
\end{equation*}
where $\mu_\alpha$ is defined in \autoref{eq:natureal-extension}. We can easily deduce that the stabilizer of $\psi_A'$ in $ \s_A$ is $ \mathrm{N}  \C_{\gol}(A)$. Since $\C_{\gol}(A)$ is abelian, we 
can extend $\psi'_A$ to a character $\psi_A{''}$ of $\mathrm{N}  \C_{\gol}(A)$. Let $\tilde{\sigma}=\ind_{\mathrm{N} \C_{\gol}(A)}^{\s_A} \psi_A''$. Then $\tilde{\sigma}$ is a $q-$dimensional irreducible representation of $\mathrm{N}  \C_{\gol}(A)$ and  $\ind_{\s_A}^{\gol}{\tilde{\sigma}}$ is a $\sns$ representation of $\gol$. By \cite{MR2584957}*{Theorem 3.1}, any $\sns$ representation of $\gol$ is of the form $\ind_{\s_A}^{\gol}\tilde{\sigma}=\ind_{\mathrm{N}  \C_{\gol}(A)}^{\gol}{\psi_A''}$ for some $A=\left(\begin{smallmatrix}
    \alpha & 1 \\
    \varpi \beta & \alpha
\end{smallmatrix}\right)\in \M_2(\cO_{\ell})$.

\section{Proof of ~\autoref*{thm:multiplicity-bounds} and ~\autoref*{thm:intermediate-steps-multiplicity-bounds}}
\label{sec:proof-multiplicity-bounds}
Let $\zol$ denote the centre of the group $\gol$. Recall that  \[
a(t, \ell) = \underset{  \rho \in \mathrm{Irr}(\gol)  }{\mathrm{Sup}} \{ m_\rho \mid  \,\, \mathrm{dim}_\mathrm{\mathbb C}(\mathrm{Hom}_{\gol}(V^t, \rho)) = m_\rho \}. 
\]
We have  $\ind_{\uol}^{\gol}\psi_t = \oplus_{\chi \in \widehat{\zol}} \ind_{\zol \uol}^{\gol} (\chi \otimes \psi_t),$ where 
\[
\chi \otimes \psi_t: \zol \uol \rightarrow \mathbb C^\times; \,\, \left(\begin{smallmatrix}
    x & 0 \\ 0 & x
\end{smallmatrix} \right)\left(\begin{smallmatrix}
    1 & y \\ 0 & 1 
\end{smallmatrix}\right) \mapsto \chi(x) \psi(\varpi^{\ell -t} y).
\]
Hence, to determine $a(t, \ell)$, it is sufficient to determine the multiplicity bounds of irreducible constituents of $\ind_{\zol \uol}^{\gol} (\chi \otimes \psi_t)$ for all $\chi \in \widehat{\zol}.$ Denote $\ind_{\zol \uol}^\gol(\chi\otimes\psi_t)$ by $V_\chi^t$. At times, we will also denote $\chi \otimes \psi_t$ by $(\chi, \psi_t)$. We fix $\lambda \in \cO_\ell$ such that $\chi(1 + \varpi^\lup x) = \psi(\varpi^\lup \lambda x)$ for all $x \in \cO_\ell$. 
Let $(\vtc)^{\text{reg}}$   and  $(\vtc)^{\text{non-reg}}$ denote the regular and the non-regular part of the representation $\vtc$. For $t < \ell$, $ \psi_t$ is also a character of $\U_{\ell-1}$. We will denote this character by $\psi_t$ itself.

 Our first result proves that the non-regular part of $\vtc$ can be understood by combining induction with the understanding of the regular part of $\ind_{\zol \uol}^{\gol} (\chi,\psi_t)$ for all $t$, $\chi$ and $\ell$.

\begin{lemma}
\label{lem: non-regular constituents}
  For $\ell \geq 2$ and $t \in [0, \ell-1]$, we have $(\vtc)^{\text{non-reg}} \cong \ind_{{\Z_{\ell-1}}{\U_{\ell-1}}}^{\G_{\ell-1}}(\bar{\chi}\otimes\psi_t) $ for a character $\bar{\chi}$ of $\Z_{\ell-1}$ such that $\bar{\chi} = \chi \circ \chi_0^2$ for some $\chi_0 \in \widehat{\zol}$. 
  
\end{lemma}
\begin{proof} 
By definition, an irreducible representation $\rho$ of $\gol$ is non-regular if and only if $\rho$ is a sub-representation of $\ind_{\K^{\ell-1}}^{\gol}(\psi_{\left(\begin{smallmatrix}  x & 0 \\ 0  & x \end{smallmatrix} \right)})$ for some $\psi_{\left(\begin{smallmatrix}  x & 0 \\ 0  & x \end{smallmatrix} \right)} \in \widehat{\K^{\ell-1}}$. We note that $$(\chi \otimes \psi_t)|_{\zol \uol \cap \K^{\ell-1}} = (\psi_A)|_{\zol \uol \cap \K^{\ell-1}},$$ where $ A= \left(\begin{smallmatrix} \lambda/2 & 0 \\ 0  & \lambda/2 \end{smallmatrix}\right).$ 
Hence
\[
(V_\chi^t)^{\mathrm{non-reg}} \cong \ind_{\zol \uol \K^{\ell-1}}^{\G_\ell}(\chi \otimes \psi_t \otimes \psi_A),  
\]
where $\chi \otimes \psi_t \otimes \psi_A(xy) = (\chi \otimes \psi_t)(x) \psi_A(y)$ for $x \in \zol \uol$ and $y \in \K^{\ell -1}.$ 
Therefore $(V^t_ \chi)^{\mathrm{non-reg}}$ is a sub-representation of $\ind_{\K^{\ell-1}}^{\gol}(\psi_A)$.
Let $\delta_A \in \mathrm{Irr}(\gol)$ be an extention of $\psi_A$ to $\gol$ given by
\[
\delta_A(g) = \psi(\tilde{\lambda}\, \det(g)),
\]
where $\tilde{\lambda} \in \cO_{\ell}$ is a lift of $\lambda.$
By \cite{MR2270898}*{Corollary~6.17}, we have 
\[
\ind_{K^{\ell-1}}^{\gol}(\psi_A) \cong \delta_A \otimes \mathrm{Ind}_{(e)}^{\gol}(\mathbf {1}) \cong \mathbb{C}[\G_{\ell-1}],
\]
where $(e)$ and $\mathbf{1}$ denote the trivial subgroup of $\gol$ and the trivial character of $(e),$ respectively.   
Therefore, 
\[
\ind_{\zol \uol\K^{\ell-1}}^{\gol}(\chi \otimes \psi_t \otimes \psi_A) \cong \ind_{{\Z_{\ell-1}}{\U_{\ell-1}}}^{\G_{\ell-1}}(\bar{\chi}\otimes\psi_t),
\]
where $\bar{\chi} = \chi \circ \delta_A^{-1}$ is a character of $\Z({\co_{\ell-1}}).$ We note that $\delta_A^{-1}= \chi_0^2$ on $\zol$ for some character $\chi_0$ of $\zol$,  hence $\bar{\chi} = \chi \circ \chi_0^2.$
\end{proof}

\begin{lemma}
		\label{type of the matrix A}
  Let $X \in \g(\cO_{\ell_1})$ be such that $\psi_X$ is a constituent of $\mathrm{Res}^{\gol}_{\K^\lup}(V^t_\chi)$, then $X$ is conjugate to  $ \left(\begin{smallmatrix}   
 a  & b \\ 
\varpi^{\ell-t} & \lambda-a
\end{smallmatrix}\right) 
$
for some $a, b \in \cO_{\ldown}$.

	\end{lemma}
	\begin{proof}  
  By Mackey's restriction formula and the definition of $\vtc$ 
\[ \mathrm{Res}_{\K^{\ell_2}}^\gol\mathrm{Ind}_{\zol \uol}^\gol({\chi},\psi_t) \cong \oplus_{g\in \K^{\ell_2}\backslash \gol/\zol \uol } \mathrm{Ind}_{{\K^{\ell_2}}\cap{(\zol \uol})^g}^{\K^{\ell_2}}({\chi},\psi_t)^g.
		\]
		Hence, upto conjugation of $A$, the character $\psi_A$ of $\K^{\ell_2}$ satisfies 
$({\chi},\psi_t)=\psi_A$ on ${\K^{\ell_2}}\cap{\zol \uol}$. 
We have
		\[
		\K^{\lup}\cap \zol  \uol=
		\left\{\begin{pmatrix}
			1+\varpi^{\lup}x&\varpi^{\lup}y\\
			0&1+\varpi^{\lup}x
		\end{pmatrix} \mid x,y \in \co_{\ell} \right\}. 
		\]
  By comparing definition of $({\chi},\psi_t)$ and $\psi_A$  

for $A = (a_{ij}) \in M_2(\cO_{\ldown}),$ 

we obtain the following: 
		\begin{gather}
			a_{11}+a_{22}=\lambda \mod(\varpi^{\ldown}),\label{eq:1a}\\ 
			a_{21}=\varpi^{\ell-t} \mod(\varpi^{\ldown}). 
			\label{eq:1c}
		\end{gather} 

Hence, our claim is proved.
	\end{proof}
For $t < \ell$, it is clear that any regular matrix $A=\left(\begin{smallmatrix}   
 a & b \\ 
\varpi^{\ell-t} & \lambda - a
\end{smallmatrix}\right)$ for $\lambda, a, b \in \Ol$ is either a $\ss$ or a $\sns$ matrix and not a cuspidal one. Therefore, the following holds. 
\begin{corollary}
\label{cor:cuspidal-multiplicity-zero=vtc}
 For $t < \ell,$ we have $\mathrm{Hom}_\gol(\rho, \vtc) = 0$ for any $\rho \in \mathrm{Irr}^{\mathrm{cus}}(\gol)$.    
\end{corollary}
Further, we observe that $A = \left(\begin{smallmatrix} a & b \\ \varpi^{\ell-t} & \lambda - a \end{smallmatrix}\right)$ is $\ss$  if and only if $\lambda-2a\neq 0 \mod (\varpi)$ and $\sns$ if and only if $\lambda-2a= 0 \mod (\varpi)$.
We now proceed to determine the multiplicity of $\ss$ and $\sns$ constituents of $V^t_\chi.$ We will consider these cases separately in the following sections. 
\subsection{ Split semisimple case}
As mentioned earlier, \autoref{thm:multiplicity-bounds} is already known for $t = \ell$, see \cite{MR4399251}. We will now assume that $0 \leq t < \ell$.  Let $\rho$ be a $ss$ representation of $\gol$. By \autoref{lem:split-semisimple from borel induction}, there exist characters $\chi_1, \chi_2$ of $\Ol^\times$  such that $\chi_1 \chi_2^{-1}$ is an injective character of $\cO_\ell^\times$ and $\rho \cong \ind_\bol^\gol(\chi_1,\chi_2)$. Let  $m_\rho = \langle\vtc,\rho \rangle_{\gol}.$ We assume that $m_\rho \neq 0.$  By Mackey's restriction formula, we have
    \begin{equation}
\label{eq:Mackey-intertwiner-SS}   
{\langle\vtc,\ind_\bol^\gol(\chi_1,\chi_2)\rangle}_{\gol}=\oplus_{g\in \bol\backslash \gol/\zol \uol}{\langle {(\chi\otimes\psi_t)}^g,(\chi_1,\chi_2)\rangle}_{\zol \uol^g\cap \bol}.
\end{equation}
Since $m_\rho \neq 0$, we must have 
\begin{equation}
\label{eqn: ss Z character equality}
    (\chi_1,\chi_2)|_{\zol}=\chi|_{\zol}.
\end{equation}

\begin{lemma}
\label{info about B_G_U}
For any $j \geq 0$, let $\rmf(j) = \max\{\ell - j, 0\}$ and $ \Delta =  \{\left(\begin{smallmatrix}
       1 & 0\\
       \varpi^jz & 1
\end{smallmatrix}\right),\left(\begin{smallmatrix}
       0 & 1\\
       1 & 0
\end{smallmatrix}\right)\mid  z \in \co_\ell^\times, 1\leq j\leq \ell \}$. Then we have
\begin{enumerate}
    \item $\bol \backslash \gol /\zol \uol \subseteq \Delta$. 
   
   \item For $j<\ell, z\in \co_\ell^\times$ and $y\in \co_\ell$, consider $g_1= \left( \begin{smallmatrix}
            1&0\\
            \varpi^jz&1
        \end{smallmatrix} \right)$ and $g_2=\left( \begin{smallmatrix}
            1&0\\
            \varpi^jz+\varpi^{2j}y&1
        \end{smallmatrix} \right)$.  Then $$\bol (g_1) \zol\uol=\bol(g_2)
\zol \uol.$$
         \item $(\zol \uol)^g\cap \bol = \{ aI_2 +  
       \left(\begin{smallmatrix}
           -\varpi^{j} xz & x\\
           0 & \varpi^{j}  xz
       \end{smallmatrix}\right) \mid  a \in \co_\ell^\times, \mathrm{val}(x) \geq \rmf(2j)  \}$   for  $g=\left(\begin{smallmatrix}
           1 & 0\\
           \varpi^j z & 1
        \end{smallmatrix}\right),$ and $(\zol \uol)^g\cap \bol = \zol$ for $g=\left(\begin{smallmatrix}
           0 & 1\\
           1 & 0
       \end{smallmatrix}\right).$
        
\end{enumerate} 
   
\end{lemma}
\begin{proof} Here (1) follows from \cite{MR2543507}. For (2) take $x=-yz^{-1}{(z+\varpi^jy)}^{-1}$, $x_2= 1+\varpi^jyz^{-1}$, $x_3=-x$ and $x_1=1-\varpi^jzx_3$. Then we have,
        $$\begin{pmatrix}
            x_1&x_3\\ 0&x_2
        \end{pmatrix}\begin{pmatrix}
            1&0\\
            \varpi^jz&1
        \end{pmatrix}\begin{pmatrix}
          1&x\\0&1\end{pmatrix}=
 \begin{pmatrix}
            1&0\\
            \varpi^jz+\varpi^{2j}y & 1
        \end{pmatrix}.$$
        This proves (2).
 Further, (3) follows from the direct computations.   
\end{proof}

    \begin{lemma} \label{lem: ss multiplicity less than 2}
For characters $\chi_1, \chi_2$ of $\Ol^\times$  such that $(\chi_1,\chi_2)\in C$ we have $$\langle \vtc ,\ind_\bol^\gol(\chi_1,\chi_2)\rangle_{\gol}\leq 2.$$
   \end{lemma}
\begin{proof}   We use \autoref{eq:Mackey-intertwiner-SS} and \autoref{info about B_G_U} to prove this result. For $g=\left(\begin{smallmatrix}
        0 & 1\\ 1 & 0
    \end{smallmatrix}\right) \in \Delta$, by $\zol \uol^g\cap \bol =\zol$ and \autoref{eqn: ss Z character equality}, we get
    \begin{equation}
          \chi\otimes\psi_t^g(x)=\chi(x)=(\chi_1,\chi_2)(x), \label{eqn: } 
    \end{equation}
    for all $x \in \co_\ell^{\times}.$ Therefore it is enough to show that among the remaining $g\in \Delta$, atmost one more can give a non-zero intertwiner. Assume that 
  $g_k=\left(\begin{smallmatrix}
           1 & 0\\
           \varpi^{j_k} z_k & 1
     \end{smallmatrix}\right)$ for $1\leq k\leq 2$ and satisfies 
       \begin{equation}
       \label{eqn: ss g_k satisfies on ZU^gk cap B}
         \chi\otimes\psi_t^{g_k} =(\chi_1,\chi_2) \text{ on } (\zol\uol)^{g_k}\cap \bol, \text{ for } k\in \{1,2\}.    
        \end{equation}
       Without loss of generality, assume that $j_1 \leq j_2.$ Consider $x = \varpi^{\ell - 2 j_1}x' \in \cO_\ell$ such that $x' \in \cO_\ell.$ Then $-\varpi^{2{j_k}}z_k^2x=0$. 
     By \autoref{info about B_G_U}, we have  $$X_k =\begin{pmatrix}
         1-\varpi^{j_k}z_kx&x\\
         0&1+\varpi^{j_k}z_kx
     \end{pmatrix}\in (\zol \uol)^{g_k}\cap \bol,$$ for $k \in \left\{ 1,2 \right\}$. We note that $\chi \otimes \psi_t^{g_1}(X_1) = \chi \otimes \psi_t^{g_2}(X_2)$. Hence by \autoref{eqn: ss g_k satisfies on ZU^gk cap B}, we must have $ (\chi_1,\chi_2)(X_1) =(\chi_1,\chi_2) (X_2)$. This is equivalent to
     $\chi_1\chi_2^{-1}\{1+(\varpi^{j_2}z_2-\varpi^{j_1}z_1)x\} =1,$ where $x = \varpi^{\ell - 2 j_1}x'$ with an arbitrary $x'$. Since $\chi_1 \chi_2^{-1}$ is injective, we obtain  $\varpi^{\ell-2{j_1}}(\varpi^{j_2}z_2-\varpi^{j_1}z_1)=0$. In other words $\varpi^{j_2}z_2-\varpi^{j_1}z_1=\varpi^{2{j_1}}y$ for some $y\in \co_\ell$.
        Now by \autoref{info about B_G_U}(2), we have $\bol g_1 \zol\uol=\bol g_2 \zol \uol.$ This proves the lemma.
     \end{proof}

\begin{proposition}
\label{prop: ss in indzu-G}
Let $\rho\in \mathrm{Irr}^{\mathrm{ss}}(\gol)$ such that $m_\rho \neq 0.$  Then  
$\mathrm{dim}_\mathbb C(\mathrm{Hom}_{\gol}(\rho, \vtc)) = 2.$
\end{proposition} 

\begin{proof}

By \autoref{sec:construction-SS} and \autoref{lem: ss multiplicity less than 2}, it is enough to prove that
$\langle\vtc,\ind_{\bol}^{\gol}(\chi_1,\chi_2)
\rangle \geq 2$.
For characters $\chi_1,\chi_2$ such that $(\chi_1,\chi_2)\in C$, let $\nu\in \co_\ell^{\times} $ such that $\chi_1\chi_2^{-1}(1+\varpi^{\lup}x)=\psi({\varpi^{\lup}\nu x})$. Consider $g=\left(\begin{smallmatrix}
    1 & 0\\
    \varpi^{\ell-t}\nu^{-1} & 1
\end{smallmatrix}\right)$. By considering $2t\leq \ell$ and $2t>\ell$ cases separately, we can easily deduce that 
 $\chi\otimes\psi_t^g=(\chi_1,\chi_2)$ on $\zol\uol^g\cap \bol$. 
This proves the claim and completes the proof of the proposition.  \end{proof}

\begin{corollary}
\label{cor:coset-form-BGU} Let $t < \ell$ and $\chi, \chi_1, \chi_2$ as in \autoref{lem: ss multiplicity less than 2}. Then $\langle (\chi \otimes \psi_t^g, (\chi_1, \chi_2) \rangle_{\zol\uol^g \cap \bol} \neq 0$ for $g\in \bol \backslash \gol / \zol\uol$ implies either $g = \left(\begin{smallmatrix}
    1 & 0 \\ 0 & 1
\end{smallmatrix}\right)$ or $g = \left(\begin{smallmatrix}
    1 & 0 \\ \varpi^{\ell - t} \nu^{-1} & 1
\end{smallmatrix}\right)$ where $\nu$ is such that $\chi_1\chi_2^{-1}(1 + \varpi^{\lup}x) = \psi(\varpi^{\lup}\nu x).$    
\end{corollary} 

We now describe the number of $\ss$ and $\sns$ constituents of $V_\chi^t.$ Let $n_{\mathrm{ss}}(V_\chi^t)$  and $n_{\mathrm{sns}}(V_{\chi}^t)$ denote the number (counted with multiplicity) of $\ss$ and $\sns$ constituents of $\vtc$.

\begin{corollary}
\label{coro:number-ss-sns-vtc}
For the representation $\vtc$ of $\gol,$ we have the following:
\begin{enumerate}
\item $n_{\mathrm{ss}} = (q-1)^2 q^{\ell-2}.$
\item $n_{\mathrm{sns}}=q^{\ell-2}(q-1).$
\end{enumerate}
    
\end{corollary}

\begin{proof}
 By \autoref{lem: non-regular constituents}, the dimension of $(\vtc)^{non-reg}$ is $q^{2\ell-4}(q^2-1)$. Recall $\text{dim}_{ss}$ and $\text{dim}_{sns}$ denote the dimensions of $\ss$ and $\sns$ representations of $\gol$. By \autoref{prop: ss in indzu-G}, all $\ss$ representations with fixed central character $\chi_1\chi_2$ are constituents of $\vtc$ with multiplicity two. Therefore using \autoref{table:numbers-dimensions}, $n_{\text{ss}}(\vtc)=2 \frac{(q-1)^3 q^{2\ell-3}}{2}/|\zol|=(q-1)^2 q^{\ell-2}$. To compute $n_{\text{sns}}(\vtc)$, we first note that 
\[
\dim(\vtc) =\dim((\vtc)^{non-reg})+n_{ss}(\vtc)\dim_{ss}+n_{sns}(\vtc) \dim_{sns} .
\]
By direct computation, using $\dim(\vtc)=q^{2\ell-2}(q^2-1)$ and \autoref{table:numbers-dimensions}, we obtain $n_{\mathrm{sns}}(\vtc)=q^{\ell-2}(q-1)$.
\end{proof}

\subsection{ Split nonsemisimple Case} 

In this section, we consider the case of $\sns$ constituents of $\vtc$. Again as mentioned earlier, we assume $0\leq t< \ell$. First, we prove several lemmas that will be useful in proving the main result.
Let $A= \left(\begin{smallmatrix}     
 a  & b \\ 
\varpi^{\ell-t} & \lambda - a
\end{smallmatrix}\right)\in \mathfrak{g}(\co_{\ldown})$. For $A$ to be a $\sns$ matrix, we must have $\lambda-2a=0 \mod (\varpi)$. Assume $\lambda-2a=\varpi^i d$ for some $d \in \co_{\ldown}^\times$ and $1\leq i \leq \ldown.$ In this subsection we will work with this $A=\left(\begin{smallmatrix}     
a  & b \\ 
\varpi^{\ell-t} & a+\varpi^i d
\end{smallmatrix}\right)\in \mathfrak{g}(\co_{\ldown})$. Consider the set   $\Sigma = \left\{ \left(\begin{smallmatrix}  1 & 0 \\ z& w \end{smallmatrix}\right), \left( \begin{smallmatrix} 0 & 1 \\ w & \varpi z' \end{smallmatrix}\right)  \right\} \subseteq \gol.$ 
\begin{lemma}
\label{lem:Z Kl2 double cosets}
 The set $\Sigma$ is an exhaustive set of  double coset representatives for the following cases:
\begin{enumerate} 
\item For $ \zol \uol \backslash \gol / \K^{\lup}.$ Further $w, z, \varpi z'$ are determined modulo $(\varpi^{\lup}).$ 
\item For $ \zol \uol \backslash \gol / \s_A$ where  $\lup\geq t .$ Further $w$, $z$ and $\varpi z'$ are determined modulo $(\varpi^{\ldown}).$
\item For $ \zol \uol \backslash \gol  / \s_A$, where  $\lup <  t .$ Further $z$, $\varpi z'$ are determined modulo $(\varpi^{\ell-t})$ and $w$ is determined modulo $(\varpi^{\ldown}).$  
\end{enumerate} 
\end{lemma}

\begin{proof} We include a proof of (3) here. The rest of the parts similarly follow by direct computations. 
Assume $g=\left(\begin{smallmatrix}
        1 & 0\\
        \varpi^{\ell-t}z & w
    \end{smallmatrix}\right)$ and $g'=\left(\begin{smallmatrix}
        1 & 0\\
        0 & w'
    \end{smallmatrix}\right)$ where $w'=w-\varpi^{\ell-t}z^2w^{-1}b-\varpi^i d z$. For proving the claim, it is enough to show that $\zol\uol g \s_A=\zol\uol g' \s_A$. For $\lup<t$, take $y=-zw^{-1}$, $x=1=ay$ and $\alpha=bzw^{-1}(w-\varpi^{\ell-t}z^2w^{-1}b-\varpi^idz)^{-1}$. Hence
\[
\left(\begin{smallmatrix}
    1 & \alpha\\
    0 & 1
\end{smallmatrix}\right)\left(\begin{smallmatrix}
        1 & 0\\
        \varpi^{\ell-t}z & w
    \end{smallmatrix}\right)
   \left( \begin{smallmatrix}
        x+ya & by\\
        \varpi^{\ell-t}y & x+ay+\varpi^idy
    \end{smallmatrix}\right)=\left(\begin{smallmatrix}
        1 & 0\\
        0 & w-\varpi^{\ell-t}z^2w^{-1}b-\varpi^idz
\end{smallmatrix}\right).
\]
This completes the proof.
\end{proof}

 \begin{lemma}
  \label{conditions with Kl2}
  For a $\sns$ matrix $A = \left(\begin{smallmatrix}   
a  & b \\ 
\varpi^{\ell-t} &  a+\varpi^i d
\end{smallmatrix}\right)$ and $g\in \Sigma$, we have
 $\chi\otimes\psi_t=\psi_A^g$ on $\zol \uol \cap \K^{\lup}$ if and only if $g=\left(\begin{smallmatrix}  
    1 & 0\\
    \varpi z & w
\end{smallmatrix}\right)$ for some $z\in \co_\ell$ and $w\in \co_\ell^\times$ and
      \begin{equation*}
      \varpi^{\ell-t}(w-1)-\varpi^2 z^2 w^{-1}b-(\varpi^i d)\varpi z=0 \mod(\varpi^{\ldown}).        \end{equation*} 
\end{lemma}

\begin{proof}
By Mackey's restriction formula,
    \begin{equation*}
          \langle \mathrm{Ind}_{ \zol \uol }^\gol((\chi\otimes\psi_t) , \ind_   {\K^\lup}^\gol(\psi_A)   \rangle  = \oplus_{g \in  \zol \uol \backslash \gol / \K^{\lup} } \langle( \chi\otimes\psi_t), \psi_A^g   \rangle_{ \zol \uol \cap \K^{\lu}}. \label{eq:mac-res-thm}
    \end{equation*}
      We first claim that  $\langle(\chi\otimes \psi_t), \psi_A^g   \rangle_{ \zol \uol \cap \K^{\lu}} = 0$ for any $g \in \Sigma$ such that $g = \left(\begin{smallmatrix} 0 & 1 \\ w & z \end{smallmatrix}\right) 
    $ and $z$ is a non-unit.
    For $X = \left(\begin{smallmatrix}  1+\varpi^{\lup}y  & \varpi^{\lup}x \\ 0 & 1+\varpi^{\lup}y   \end{smallmatrix}\right)\in \zol \uol \cap K^{\lup}$ and $g = \left(\begin{smallmatrix}  0 & 1 \\ w & z \end{smallmatrix}\right)$, 
   \begin{equation*}
   \label{eqn: chi tensor psi of k-l2 intertwiner}
        (\chi\otimes\psi_t) (X ) = \psi ( \varpi^{\ell - t + \lup}  x+\varpi^{\lup}\lambda y)
   \end{equation*} 
   
   and  
   $$
   \psi_A^g(X) = \psi( \varpi^{\lup} x (bw + (\varpi^i d) z - \varpi^{\ell -t} w^{-1} z^2 +\varpi^{\lup}\lambda y) ).  
   $$
    Hence $\chi \otimes \psi_t = \psi_A^g$ on $\zol \uol \cap \K^{\lup}$ implies
    \begin{gather}
    \psi(\varpi^\lup x ( \varpi^{\ell -t} (-1- w^{-1} z^2 )  + bw + (\varpi^i d)z  ))   =   1 \label{eq:conj-eq-3}, 
    \end{gather}
    for all $x \in \cO_\ell$.  Since $A$ is regular we must have $b$ is a unit for $\ell> t$. Since $w$ is a unit and  $\psi$ is a non-degenerate character, we observe that \autoref{eq:conj-eq-3} does not hold for all $x \in \cO_\ell$. This completes the proof of the claim. We next focus on $g \in \Sigma$ such that $g =  \left(\begin{smallmatrix}   1 & 0 \\ z & w  \end{smallmatrix}\right).$ 
    For $X = \left(\begin{smallmatrix} 1+\varpi^{\lup}y & \varpi^{\lup}x \\ 0 & 1+\varpi^{\lup}y  \end{smallmatrix}\right) \in \zol \uol \cap \K^{\lup}$, we have 
    \begin{equation*}
    \label{eqn: psi-A-g of k-l2 intertwiner}
        \psi_A^g(X) = \psi(\varpi^\lup((-z^2w^{-1}b
 +\varpi^{\ell-t}w-(\varpi^i d)z)x+\lambda y)). 
    \end{equation*}
 
 Hence $\psi_A^g(X) = \chi\otimes \psi_t(X)$ if and only if 
 \begin{gather*}
 \varpi^{\ell-t}(w-1)- z^2 w^{-1}b-(\varpi^i d) z=0 \mod(\varpi^{\ldown}). \end{gather*}

  We notice that the above equation is satisfied only if $z$ is a non-unit. This proves the result.
\end{proof}

\begin{lemma}
\label{lem: number-of-non-conjugate-sns-matrices}
For $2 \leq \ell-t$ and fixed $\lambda\in \Ol$, let $\mathcal{S}=\{\left(\begin{smallmatrix}
   a & b\\
      \varpi^{\ell-t} & \lambda-a
\end{smallmatrix}\right)\in \mathfrak{g}(\co_{\ell}) \mid \lambda-2a=0 \mod(\varpi), \,\, b \in \co_{\ell}^\times\}.$ 
Let $\mathcal{M}$ be a set of all pairwise non-similar matrices in $\mathcal{S}.$  Then
   $|\mathcal{M}|\leq q^{\ell-2}$.
\end{lemma}
\begin{proof}
Let $A =\left( \begin{smallmatrix} a & b \\ \varpi^{\ell-t} & \lambda-a \end{smallmatrix} \right)\in \mathcal{S}.$ Then $D=\left(\begin{smallmatrix}
    \lambda/2 & 1\\
    \lambda^2/4- \mathrm{det}(A) &  \lambda/2
\end{smallmatrix}\right)\in \mathfrak{g}(\co_{\ell})$ is conjugate to $A$ because $A$ and $D$ are regular matrices in $\mathfrak{g}(\co_{\ell})$ with $\text{trace}(A)=\text{trace}(D)$ and $\text{det}(A)=\text{det}(D)$. Assume $\lambda-2a=\varpi^i d $ for some $1\leq i\leq \ldown$. Then $\lambda^2/4-\text{det}(A) = \varpi^2(\frac{\varpi^{2i-2}d^2}{4}+\varpi^{\ell-t-2}b).$ 
For $x=(\frac{\varpi^{2i-2}d^2}{4}+\varpi^{\ell-t-2}b)$, the matrix $D$ can be written as $ \left(\begin{smallmatrix}
    \lambda/2 & 1\\
    \varpi^2 x &  \lambda/2
\end{smallmatrix}\right)$,
where $x\in \Ol$. Since $\lambda$ is fixed, the number of non-conjugate matrices of the form $D$ is at most $ q^{\ell-2}$. This proves the lemma.
\end{proof}

\begin{thm}
\label{cor: SNS in induction Z}
  The following are true for any $\rho\in \mathrm{Irr}^{\sns}(\gol):$
  \begin{enumerate}
  \item For $\ell-1 \leq t$, we have $\langle V^t, \rho\rangle_{\gol} \leq 1$. 
 \item For $t< \ell-1$, the intertwiner $\langle V^t, \rho\rangle_{\gol}$ depends on the cardinality of the residue field.  \end{enumerate} 
    
\end{thm}

\begin{proof} 
For $t = \ell$, this result follows from \cite{MR4356279}. We now focus on $ t \leq \ell-1.$ Let $\rho$ be a $\sns$ representation of $\gol$ such that $\rho$ is a constituent of $\vtc$. Then the following hold for $\rho$:
\begin{enumerate} 
\item By \autoref{type of the matrix A}, there exists $A= \left(\begin{smallmatrix}   
 a  & b \\ 
\varpi^{\ell-t} & a+ \varpi^i d  
\end{smallmatrix}\right)\in \mathfrak{g}(\co_{\ldown})$ 
such that $2a+\varpi^i d=\lambda \mod (\varpi^{\ldown})$ 
and $\rho$ lies above the character $\psi_A$ of 
$\K^\lup$.  
\item By \autoref{sec:E.construction},  there exists unique $\tilde{\sigma} \in \mathrm{Irr}(\s_A \mid \psi_A)$ of dimension $q^{\lup - \ldown}$ such that $\rho \cong  \ind_{\s_A}^{\gol}(\tilde{\sigma})$. 
\end{enumerate}
  By Mackey's restriction formula,
  \begin{equation}
\label{eq:intertwiner-numbers}
{\langle V_\chi^t, \rho \rangle}_{\gol}  = {\langle \ind_{\zol \uol }^\gol(\chi\otimes\psi_t), \ind_{\s_A}^\gol \tilde{\sigma} \rangle}_{\gol}  =  \oplus_{g \in \Sigma_A} \langle \chi\otimes\psi_{t}^g , \tilde{\sigma}\rangle_{({\zol \uol })^g \cap \s_A},
\end{equation} 
where $\Sigma_A =\zol \uol / \gol \backslash \s_A.$ 
 From \autoref{conditions with Kl2}, only coset representatives that may give a non zero intertwiner between $\chi \otimes \psi_t^g$ and $\tilde{\sigma}$ on $(\zol \uol)^g \cap \s_A$ are of the form $g=\left(\begin{smallmatrix}
    1 & 0\\
    \varpi^j z &  v
\end{smallmatrix}\right)\in \Sigma_A$ with  $j \leq  \ldown$. 
We now separately consider the cases $t = \ell-1$ and $t< \ell-1$. 

\vspace{.2 cm}
{\bf Case~1: $\mathbf{t = \ell-1}$:} By \autoref{lem:Z Kl2 double cosets}(3), any $g=\left(\begin{smallmatrix}
    1 & 0\\
    \varpi z &  v
\end{smallmatrix}\right) \in \Sigma_A$ for $t = \ell-1$ can be further reduced to the form $g_w := \left(\begin{smallmatrix}
    1 & 0\\
    0 &  w
\end{smallmatrix}\right)$. Let $\Sigma_A' = \{ g_w \mid w \in \cO_\ell^\times\}.$
We note that $(\zol \uol)^g\cap \s_A=\zol \uol\cap \s_A$ for all $g\in \Sigma_A'$ and $  \zol \uol \cap \s_A = \zol \uol (\varpi^{\ld-1} \cO_\ell).$ 
We now consider $\ell$ even and odd separately. 

\vspace{.2 cm}
{\bf $\ell$-even:} \,\, 
Here $\tilde{\sigma} = \widetilde{\psi_A}$ for some extension of $\widetilde{\psi_A}$ of $\psi_A$ to $\s_A$. By ${ \langle \rho, \vtc \rangle }_{\gol} \neq 0$ and  \autoref{eq:intertwiner-numbers}, there exists $g_w \in \Sigma_A'$ such that 
\begin{equation*} 
    \chi\otimes\psi_{\ell-1}^{g_{w}}(x) = \widetilde{\psi_A}(x), \,\, \mathrm{for \,\, all}\,\, x \in U(\varpi^{\ldown-1} \co_{\ell}) \label{eqn: l even t=l-1 , equality for g}. 
\end{equation*}
Now suppose there exists  $g_{w'}=\left(\begin{smallmatrix}
    1 & 0\\
    0 &  w'
\end{smallmatrix}\right) \in \Sigma_A^{'}$ such that
\begin{equation*}
    \chi\otimes\psi_{\ell-1}^{g_w} (x) = \chi \otimes \psi_{\ell-1}^{g_{w'}}(x),  \mathrm{\,\, for \,\, all} \,\, x \in U(\varpi^{{\ldown}-1} \cO_\ell).    \label{eqn: l even t=l-1 ,equality for g'}
\end{equation*}
Then we must have $\psi(\varpi^{\ld} xw) = \psi(\varpi^{\ld} xw')$ for all $x \in \cO_{\ell}$, which gives, $w=w' \mod(\varpi^{\ld})$. This implies $\zol \uol g_w \s_A = \zol \uol  g_{w'} \s_A,$ that is both $g_w$ and $g_{w'}$ represent the same double cosets. Since $g_{w'}$ varies over the coset representatives $\zol \uol \backslash \gol / \s_A$ in \autoref{eq:intertwiner-numbers} with $\langle \chi \otimes \psi_t^{g_{w'}}, \tilde{\sigma} \rangle_{(\zol \uol)^g_{w'} \cap \s_A} \neq 0$, we obtain that $\rho$ occurs with multiplicity one as a constituent of $V_\chi^t$.

\vspace{.2 cm}
{\bf $\ell$-odd:} \,\,
For $\ell$ odd, let $R_{\tilde{A}}$ denote the radical associated to $\psi_A $ (see \autoref{O.construction}).  We note that $\U(\varpi^{\ld} \cO_\ell) \subseteq R_{\tilde{A}}$. Therefore,  by   $\tilde{\sigma}|_ {\U(\varpi^{\ld} \cO_\ell)} = \widetilde{\psi_A}^{\oplus q}$ for some extension $\widetilde{\psi_A}$. Hence $ {\langle \vtc, \ind_{\s_A}^\gol \tilde{\sigma} \rangle }_{\gol}\neq 0$ implies there exists $g_w = \left(\begin{smallmatrix}  1 & 0 \\ 0 & w \end{smallmatrix}\right)$ such that 
\begin{equation} 
\chi\otimes\psi_{\ell-1}^{g_w}(x) = \widetilde{\psi_A}(x); \,\, \forall x \in \U(\varpi^{\ldown} \cO_\ell) \label{eqn: l odd t=l-1, equality for gw}. 
\end{equation} 

By the same arguments as given in $\ell$-even, $g_w$ is the unique coset representative in $\Sigma_A^{'}$ satisfying \autoref{eqn: l odd t=l-1, equality for gw}.   To complete the proof in this case, we need to further prove that  
\[
\langle \chi \otimes \psi_{\ell-1}^{g_w},\tilde{\sigma}  \rangle_{(\zol \uol)^{g_w} \cap \s_A} \leq 1.
\]
 By \autoref{prop:multiplicity-free-SA-representation}, the representation $\tilde{\sigma} |_{\U_A}$ is multiplicity free. For $t = \ell-1$, we have $\zol \U(\varpi^{\ldown-1}\cO_\ell) = (\zol \uol)^{g_w} \cap \s_A$. Therefore $\langle \chi \otimes \psi_{\ell-1}^g,\tilde{\sigma}  \rangle_{(\zol \uol)^{g_w} \cap \s_A} \leq 1$. This completes the proof for $t = \ell-1$.

\vspace{.3cm} 
{\bf Case~2: $\mathbf{ t < \ell-1}:$} For this case, we prove that for large $q$,  there exists a $\sns$ representation $\rho$ such that $\dim_{\mathbb C}(\mathrm{Hom}_{\gol}(\rho, \vtc))$ is $q$-dependent. To prove this, we show that the number of distinct $\sns$ constituents of $\vtc$ is at most $q^{\ell -2}$. Comparing this with $n_{\text{SNS}}(\vtc) = (q-1)q^{\ell-2}$ from \autoref{coro:number-ss-sns-vtc}, we observe that there must be some $\sns$ constituent of $\vtc$ with multiplicity strictly greater than one (recall $q$ is odd). Moreover, in view of given $n_{\text{SNS}}(\vtc)$ and considering large $q$, we must have multiplicity dependent on $q$ in general. We now proceed to prove the following claim:

\vspace{.2cm} 
\noindent {\bf Claim:} The number of distinct $\sns$ constituents of $\vtc$ is at most $q^{\ell -2}.$

\noindent {\bf Proof of the Claim:} We separately consider the cases of $\ell$ even and $\ell$ odd. First assume that $\ell$ is even, that is $\ldown = \lup$. By \autoref{type of the matrix A}, any $\sns$ constituents of  $\vtc$ lies above $\psi_A$ for $A$ of the form $\left(\begin{smallmatrix} a & b \\ \varpi^{\ell-t} & \lambda-a  \end{smallmatrix}\right) \in \mathfrak{g}(\cO_{\ldown})$, where $\lambda$ is determined by $\chi$, $\lambda - 2a \in (\varpi)$, and $b \in \Ol^\times$. By \autoref{lem: number-of-non-conjugate-sns-matrices}, the number of such non-conjugate matrices is atmost $q^{\ldown-2}.$ By the construction as mentioned in \autoref{sec:E.construction}, the number of $\sns$ representations lying above each such $\psi_A$, with a fixed central character, is atmost $q^{\ldown}.$ Hence the total number of distinct $\sns$ representations of $\vtc$ is at most $q^{\ell-2}.$ This proves the claim for $\ell$ even.

We now consider the case of $\ell$ odd. Let $\rho$ be a $\sns$ representation such that $\rho$ is a constituent of $\vtc$ and $\rho \cong \ind_{\s_A}^{\gol}\tilde{\sigma}$. By definition of $\vtc$, $\tilde{\sigma}$ and Frobenius reciprocity, we must have $\tilde{\sigma}_{\mid {\K^\lup}} \cong \psi_A^{\oplus q}$. Hence it is enough to determine the number of non-conjugate $A$ and number of distinct irreducible constituents lying above $\psi_A$ and $\chi$. By \autoref{lem: number-of-non-conjugate-sns-matrices}, the number of such non-conjugate matrices is atmost $q^{\ldown-2}$. Further, by construction as given in \autoref{O.construction}, the number of distinct irreducible representations lying above $\psi_A$ and $\chi$ is $q^{\lup}$. Hence the maximum number of distinct $\sns$ constituents of $\vtc$ is $q^{\ell -2}.$ This completes the proof of the claim. 
 \end{proof}
\begin{proof}[Proofs of \autoref{thm:multiplicity-bounds} and \autoref{thm:intermediate-steps-multiplicity-bounds}]
Proof of these results directly follow from \autoref{lem: non-regular constituents}, \autoref{cor:cuspidal-multiplicity-zero=vtc}, \autoref{prop: ss in indzu-G} and \autoref{cor: SNS in induction Z}. 
\end{proof}

\section{Proof of \autoref*{thm:borel-strong-Gelfand-pair}}
\label{sec:proof-borel-strong-Gelfand-pair}
Recall that $\bol$ denotes the group of all upper-triangular matrices in $\gol$ and $\uol$ is the subgroup of $\bol$ consisting  of all unipotent matrices. Every character of $\uol$, upto conjugation by $\bol$, is of the form 
\[ \psi_t: \uol  \rightarrow \mathbb C^\times; \,\,\psi_t(\begin{pmatrix} 1 & x \\ 0 & 1 \end{pmatrix} = \psi(\varpi^{\ell -t} x),
\]
for $0 \leq t \leq \ell.$ We observe that $\ztol \uol$, where $\ztol = 
\{\mat{x}{0}{0}{x+\varpi^t y} : x,y \in \co_{\ell}\}\cap \GL_2(\co_{\ell})$, is the stabilizer of $\psi_t$ in $\bol.$

Let $\chi$ be a one dimensional representation of $\zol$, the center of $\gol$, and let $\tilde{\chi}$ be an extension of $\chi$ to $\ztol$.  The following defines a one-dimensional representation of $\ztol \uol$.
	$$ (\tilde{\chi}, \psi_t) : \ztol \uol \rightarrow \mathbb{C}^\times;  \, \,  \,\,   xy \mapsto \tilde{\chi}(x)\psi_t(y),$$ 
 for $x \in \ztol$ and $y \in \uol.$ Every character of $\ztol \uol$ lying above $\psi_t$ is of the form $(\tilde{\chi}, \psi_t)$ for some $\tilde{\chi} \in \widehat{\ztol}$, an extension of ${\chi}\in \widehat{\zol}$. Hence ~\autoref{thm:borel-strong-Gelfand-pair} follows from Clifford's theory and the following result. 
 \begin{thm}
 \label{thm:multiplicity-one-wtc}
 The $\gol$-representation space  $\mathrm{Ind}_{\ztol \uol }^{\gol}(\tilde{\chi},\psi_t)$ is multiplicity free for every representation $(\tilde{\chi},\psi_t)$ of $\ztol \uol$.
\end{thm}
We will now focus on proving the above result. For $t = 0$ and $t = \ell$, this follows  from \cite{MR0338274} and \cite{MR4399251}, respectively. The restriction of $(\tilde{\chi}, \psi_t)$ to $\zol \uol$ is denoted by $(\chi, \psi_t)$. We denote $\mathrm{Ind}_{\ztol \uol }^{\gol}(\tilde{\chi},\psi_t)$ by $\wtc$ and as earlier $\mathrm{Ind}_{\zol \uol }^{\gol}(\chi,\psi_t)$ by $\vtc$. Note that $\wtc$ is a subrepresentation of $\vtc$. Therefore, the result for $t = \ell-1 $ follows from \autoref{thm:multiplicity-bounds}. Hence, we will assume $1 \leq t \leq \ell-2$ for the rest of this section. 

As $\tilde{\chi}$ is an extension of $\chi$ to $\ztol$,  there exists ${\chi'} \in \widehat{1 + \varpi^t \cO_\ell}$ such that
\[
\tilde{\chi}: \ztol \rightarrow \mathbb C^\times; \,\,\, \mat{x}{0}{0}{x} \mat{1}{0}{0}{1+\varpi^{t} y} \mapsto \chi(x) {\chi'}(1 + \varpi^t y).
\]
In this case, we denote $\tilde{\chi}$ by  $(\chi, {\chi'})$ and $(\tilde{\chi}, \psi_t)$ by $(\chi, {\chi'}, \psi_t).$ 
 Let $s =\max\{t,\lup\}$ and $ s' = \min\{t,\lup\}$. 
There exists $\lambda, \lambda' \in\co_\ell$ such that $\chi(1+\varpi^{\lup} x)=\psi(\varpi^{\lup}\lambda x)$ and ${\chi'}(1+\varpi^{s}x)=\psi(\varpi^{s}\lambda' x)$ for $x \in \co_\ell$.   
The following result has a proof parallel to \autoref{lem: non-regular constituents} and \autoref{type of the matrix A}, hence we omit it here. 
  \begin{lemma}
\label{lem: non-regular constituents-W-type of the matrix A-wtc}
\begin{enumerate}
  \item For $t \in [1, \ell-2]$, we have $(\wtc)^{\text{non-reg}} \cong \ind_{\Z^t_{\ell-1}\U_{\ell-1}}^{\G_{\ell-1}}(\tilde{\chi}_1, \psi_t) $ for some character $\tilde{\chi}_1$ of $\Z^t_{\ell-1}$.
  \item Let $X \in \M_2(\cO_{\ell_1})$ be such that $\psi_X$ is a constituent of $\mathrm{Res}^{\gol}_{\K^\lup}(\wtc)$, then $X$ is conjugate to 
 
   $ \left(\begin{smallmatrix}   
 \lambda-\lambda'-\varpi^{\ell-s}\beta & b \\ 
\varpi^{\ell-t} & \lambda'+\varpi^{\ell-s}\beta
\end{smallmatrix} \right)
$
for some $\beta, b \in \cO_{\ldown}$. 
  \end{enumerate} 
\end{lemma}
In view of the above lemma, to prove \autoref{thm:multiplicity-one-wtc}, it is enough to prove that every regular representation of $\gol$ has multiplicity at most one as a subrepresentation of $\wtc.$ For $\ell = 1$, it is well known. Hence, we will now onwards assume that $\ell \in [2, \ell-2]$. For the cuspidal representations of $\gol,$  it follows from \autoref{cor:cuspidal-multiplicity-zero=vtc}. We now proceed to focus on $\ss$ and $\sns$ representations of $\gol.$

    \begin{lemma}
  For any $\ss$ representation $\rho$, we have $\dim_{\mathbb C}(\mathrm{Hom}_\gol(\rho, \wtc)) \leq 1.$ 
\end{lemma}

\begin{proof}  
From \autoref{lem: non-regular constituents-W-type of the matrix A-wtc}, we note that $\dim_{\mathbb C}(\mathrm{Hom}_{\gol}(\rho, \wtc)) \neq 0$ for a $\ss$ representation $\rho$ of $\gol$ implies $\lambda - 2\lambda' \notin (\varpi).$
Since $\wtc\subseteq \vtc$, by \autoref{lem: ss multiplicity less than 2}, we have $\dim_{\mathbb C}(\mathrm{Hom}_{\gol}(\rho, \wtc)) \leq 2.$
By Mackey's restriction formula,
\[
\langle\ind_{\ztol \uol}^\gol(\tilde{\chi},\psi_t),\ind_\bol^\gol(\chi_1,\chi_2)\rangle_{\gol} =\oplus_{g\in \bol\backslash \gol/\ztol \uol}{\langle {(\tilde{\chi}, \psi_t)}^g,(\chi_1,\chi_2)\rangle}_{(\ztol \uol)^g\cap \bol}.
\]
From \autoref{cor:coset-form-BGU}, only cosets in $\ztol \uol \backslash \gol /\bol$ that may give a non zero intertwiner are $g_1=\left(\begin{smallmatrix}
    0 & 1\\
    1 & 0
\end{smallmatrix}\right)$ and $g_2=\left(\begin{smallmatrix}
    1 & 0\\
    \varpi^{\ell-t}\delta^{-1} & 1
\end{smallmatrix}\right),$ where $\delta$ is as given in \autoref{cor:coset-form-BGU}.
We claim that both $g_1$ and $g_2$ can not simultaneously give non zero intertwiner in this case. This will prove our result. We prove our claim by contradiction. Suppose
\begin{gather}
(\chi,{\chi'},\psi_t)^{g_i}=(\chi_1,\chi_2) \text{ on }(\ztol \uol)^{g_i}\cap \bol,  \label{eq:2e}
\end{gather}
for $1 \leq i \leq 2$. We note that $(\ztol \uol)^{g_1}\cap \bol=
       \ztol$ and  
    \[
   (\ztol \uol)^{g_2} \cap \bol =  \{ aI_2 + \varpi^{\max\{ 2t-\ell, 0\}} 
       \left(\begin{smallmatrix}
           -\varpi^{\ell-t} x\delta^{-1} & x\\
           0 & \varpi^t d + \varpi^{\ell-t}  x\delta^{-1}
       \end{smallmatrix}\right) \mid  x,d \in \co_\ell, a \in \co_\ell^\times   \}.   
       \]
By considering suitable elements of $(\ztol \uol)^{g_i} \cap \bol$ and using \autoref{eq:2e}, we obtain 
\begin{gather*}
    {\chi'}(1+\varpi^td)=\chi_1(1+\varpi^td) = \chi_2 (1 + \varpi^t d) \,\,\mathrm{and}  \,\,
    \chi(c)=\chi_1\chi_2(c),
\end{gather*}
for all $d \in \co_\ell$ and $c\in \co_\ell^{\times}$.  Therefore we must have $ \chi(1+\varpi^td)= {\chi'}^{2}(1+\varpi^td).
$
 This implies $\lambda-2\lambda'=0 \mod (\varpi)$, a contradiction to our assumption that $\lambda-2\lambda'\neq 0 \mod (\varpi)$. This completes the proof of our result. 
 \end{proof}  

 For $t \leq \ell$ and $i \leq \ldown$, define  $f(t,i) = \begin{cases}  
t-\max\{t-i, 0\}, & \text{ for } \lup > t; \\
\ldown, & \lup \leq t \text{ for } \text{ and }  i+t\geq\ell; \\ 
t-\lup +i, & \lup \leq t \text{ for } \text{ and } i+t<\ell. 
\end{cases}$

We now focus on $\sns$ constituents in $\wtc$. For  $ A=\left(\begin{smallmatrix}   
 \lambda-\lambda'-\varpi^{\ell-s}\beta & b \\ 
\varpi^{\ell-t} & \lambda'+\varpi^{\ell-s}\beta
\end{smallmatrix} \right)
$ to be $\sns$, we must have $2\lambda'-\lambda+2\varpi^{\ell-s}\beta=\varpi^i d$ for some $1\leq i \leq \ldown$ and $d\in \co_{\ldown}^\times$. Hence we will use $A=\left(\begin{smallmatrix}
          a  & b \\ 
\varpi^{\ell-t} & a+\varpi^i d
      \end{smallmatrix} \right) $ from now on and use the conditions $a=\lambda-\lambda'-\varpi^{\ell-s}\beta$ and $2\lambda'-\lambda+2\varpi^{\ell-s}\beta=\varpi^i d$ wherever needed.
  \begin{lemma}
      \label{lem: wtc SNS-coset equality}
      Fix a $\sns$ matrix $A =  \left(\begin{smallmatrix}
          a  & b \\ 
\varpi^{\ell-t} & a+\varpi^i d
      \end{smallmatrix} \right) 
\in \M_2(\co_{\ldown})$. For $g=\left(\begin{smallmatrix} 1 & 0 \\ 0 & w  \end{smallmatrix}\right)$ and $g'=\left(\begin{smallmatrix} 1 & 0 \\ 0 & w'  \end{smallmatrix}\right)$ with $
    w -w' \in (\varpi^{f(t,i)}) $, we have  $(\ztol \uol) g \s_A=(\ztol \uol) g'\s_A$.
  \end{lemma}
  \begin{proof}
We have different cases depending on the function $f(t,i)$ and we will assume that $w'=w+\varpi^{f(t,i)}{w_0}$ for some ${w_0}\in \co_{\ell}$.

\noindent \textit{ Case 1: $(\lup> t)$}
For this case $\ell-t>\ldown$, therefore $A =  \left(\begin{smallmatrix} 
          a& b \\
          0 & a+\varpi^id
      \end{smallmatrix} \right).$ 

      \begin{enumerate}
          \item For $i\leq t$, we have 
             \begin{equation*}
      \label{eq:a1}
         \begin{pmatrix}
        1 & 0\\
        0 & 1+\varpi^t{w^{-1}w_0}
      \end{pmatrix}
      \begin{pmatrix}
          1 & 0\\
          0 & w
      \end{pmatrix}= \begin{pmatrix}
          1 & 0\\
          0 & w'
      \end{pmatrix}. 
      \end{equation*}
      \item For $i>t$, take $y=w^{-1}{w_0}d^{-1}$ and $\alpha w=-b y{(1+\varpi^id y)}^{-1}.$ We get
\begin{equation}
\label{eq: 2nd equation }
\begin{pmatrix}
    1 & \alpha \\
    0 & 1
\end{pmatrix}
    \begin{pmatrix}
        1 & 0\\
        0 & w
    \end{pmatrix}\begin{pmatrix}
        1 & b y\\
        0 & 1+\varpi^i dy
    \end{pmatrix}=\begin{pmatrix}
        1 & 0\\
        0 & w'
    \end{pmatrix}.
\end{equation}

      \end{enumerate}
      
\noindent   \textit{Case 2: $(\lup\leq t \text{ and } i+t\geq \ell)$}   For this case $f(t, i) = \ldown.$ Hence, the result follows because the cosets are determined modulo $(\varpi^{\ldown})$.
   
\noindent  \textit{Case 3: $(\lup\leq t \text{ and } i+t< \ell)$}  Take  $\gamma=-d^{-1}{w_0}w^{-1}$, $y=-\varpi^{t-\lup}\gamma$ ,  $x=1-ay$, and $\alpha=-b w^{-1}y(x+y(a+\varpi^i d))^{-1}$. Then we have
\begin{equation}
\label{eq:a3}
\begin{pmatrix}
    1 & \alpha\\
    0 & 1
\end{pmatrix}
\begin{pmatrix}
    1 & 0\\
    0 & w
\end{pmatrix}
\begin{pmatrix}
    x+ay & b y\\
    \varpi^{\ell-t}y+\varpi^{\ldown} \gamma & x+(a+\varpi^i d)y
\end{pmatrix}=
\begin{pmatrix}
    1 & 0\\
    0 & w'
\end{pmatrix}.
\end{equation}
  Hence the claim follows.  \end{proof}

       The following is parallel to \autoref{lem:Z Kl2 double cosets} for the current case. It follows by direct computations, so we omit the proof here. Consider the set   $\Sigma_t = \left\{ \left(\begin{smallmatrix}  1 & 0 \\ z& w \end{smallmatrix}\right),  \left(\begin{smallmatrix} 0 & 1 \\ w & \varpi z' \end{smallmatrix}  \right) \right\} \subseteq \gol$. The following lemma lists important information regarding $\Sigma$ that we use without explicitly mentioning it. 
\begin{lemma}
\label{lem:Zt-Kl2 double cosets}
Let $A= \left(\begin{smallmatrix}     a  & b \\ 
\varpi^{\ell-t} & a+\varpi^i d
\end{smallmatrix}\right)\in \mathfrak{g}(\co_{\ldown})$. Then $\Sigma$   is an exhaustive set of  double coset representatives for the following: 
\begin{enumerate}
  \item For $ \ztol \uol \backslash \gol / \K^{\lup},$ where $z$ and $\varpi z'$ are determined modulo $(\varpi^{\lup})$ and $w$ is determined modulo $(\varpi^{s'})$ where $s'= \mathrm{min}\{t, \lup\}$.
\item For $ \ztol \uol \backslash \gol / \s_A$ with  $\lup\geq t $,  where $z$ and $\varpi z'$ are determined modulo $(\varpi^{\ldown})$ and $w$ is determined modulo $(\varpi^{s'})$ where $s'= \mathrm{min}\{t, \ldown\}$.
\item For $ \ztol \uol \backslash \gol  / \s_A$ with  $\lup <  t $, where $z$ and $\varpi z'$ are determined modulo $(\varpi^{\ell-t})$ and $w$ is determined modulo $(\varpi^{s'})$ where $s'= \mathrm{min}\{t, \ldown\}$.
\end{enumerate}

\end{lemma}
 
In this section, whenever $t$ is clear from the context, we shall also denote $\Sigma_t$ by $\Sigma$ itself.

  \begin{lemma}
  \label{lem:coset-representative-form-wtc}
 Suppose $s=\mathrm{max}\{t,\lup\}$. For a $\sns$ matrix $A = \left(\begin{smallmatrix}   
a  & b \\ 
\varpi^{\ell-t} & a+\varpi^i d
\end{smallmatrix}\right)$, we have      $(\tilde{\chi}, \psi_t)=\psi_A^g$ on $\ztol \uol \cap \K^{\lup}$ for some $g \in \Sigma$ if and only if $g=\left(\begin{smallmatrix}  
    1 & 0\\
    \varpi z & w
\end{smallmatrix}\right)$ for some $z\in \co_\ell$ and 
\begin{gather*}
    \varpi zw^{-1} b  =  0 \mod(\varpi^{\ell -s}), \\ 
    \varpi^{\ell-t}(w-1) -\varpi^2 z^2  w^{-1} b -(\varpi^i d)\varpi z =  0 \mod (\varpi^{\ldown}).  
\end{gather*}
In particular $\varpi z \in (\varpi^{\ell - s})$. 
  \end{lemma}
\begin{proof}

We use an expression similar to \autoref{eq:mac-res-thm} for $\wtc$. 
Since $\wtc\subseteq \vtc$, only $g\in \Sigma$ that can give $(\tilde{\chi}, \psi_t)(X)=\psi_A^g(X)$ for $X \in \ztol \uol \cap \K^\lup$ are  such that $g=\left(\begin{smallmatrix}
      1 & 0 \\ \varpi z & w
  \end{smallmatrix}\right)$.
For $X= \left(\begin{smallmatrix}  1+\varpi^{\lup }y  & \varpi^{\lup }x \\ 0 & 1+\varpi^{\lup }y +\varpi^s\alpha   \end{smallmatrix}\right) \in \ztol \uol \cap \K^{\lup}$, we have
  \begin{gather*}
      \psi_A^g(X)=\psi(\varpi^{\lup}(2a+\varpi^i d)y+\varpi^{s-\lup}\alpha(a+\varpi^i d-bw^{-1}\varpi z)+(-\varpi^i d \varpi z-b \varpi^2 z^2 w^{-1}+\varpi^{\ell-t}w)x),
       \end{gather*}
and 
\[
(\tilde{\chi}, \psi_t)(X)=\psi(\varpi^{\lup}(\varpi^{\ell-t}x+\lambda y+\varpi^{s-\lup}\lambda'\alpha)).
\]
By comparing coefficients of $x,y$ and $\alpha$, we obtain
 \begin{gather*}
 2a+\varpi^i d=\lambda \mod(\varpi^{\ldown}), \\
a+\varpi^i d-\lambda'- w^{-1} \varpi z b =0 \mod (\varpi^{\ell-s}),\\ 
  \varpi^{\ell - t}  ( w - 1) - w^{-1} \varpi^2 z^2 b -(\varpi^i d)\varpi z = 0 \mod (\varpi^{\ldown}).
  \end{gather*}
  Now using the conditions $a=\lambda-\lambda'-\varpi^{\ell-s}\beta$ and $2\lambda'-\lambda+2\varpi^{\ell-s}\beta=\varpi^i d$ in above equations we get our result.
  \end{proof}
The following lemma, its proof is by direct computations, gives the description of the group $\ztol \uol \cap \s_A$. 
\begin{lemma}
\label{lem: ZtU intersection SA}
    Let $A=\left(\begin{smallmatrix}
        a & b\\
        \varpi^{\ell-t} & a+\varpi^i d
    \end{smallmatrix}\right)\in \mathfrak{g}(\co_{\ldown})$, then  
    \[
     \ztol\uol \cap \s_A=
     \begin{cases}
        \{ \left( \begin{smallmatrix}
        x & \varpi^{max\{0,t-i\}}y\\
        0 & x+\varpi^{i+max\{0,t-i\}} d y+\varpi^{\ldown}c 
        \end{smallmatrix}\right)\}, & \text{ if } \lup > t;  \\
          \{ \left( \begin{smallmatrix}
        x & \varpi^{t-{\lup}}y\\
        0 & x+\varpi^t c 
        \end{smallmatrix}\right)\}, & \text{ if } \lup \leq  t  \text{ and } i\leq \ldown \text{ and } i+t \geq \ell; \\
          \{ \left( \begin{smallmatrix}
        x & \varpi^{\ldown-i}y\\
        0 & x+\varpi^t c 
        \end{smallmatrix}\right)\}, & \text{ if } \lup \leq t  \text{ and } i<\ldown \text{ and } i+t < \ell.
    \end{cases}
    \]
    
\end{lemma}

\begin{proposition}
For any $\rho \in \mathrm{Irr}^{\sns}(\gol)$, we have $\dim_{\mathbb C}(\mathrm{Hom}_{\gol}(\rho, \wtc)) \leq 1.$ 
\end{proposition}
\begin{proof}
Let $\tilde{\sigma}$ be an irreducible representation of $\s_A$ lying above $\psi_A$ such that $\rho \cong \ind_{\s_A}^{\gol}\tilde{\sigma}$. By definition of $\wtc$ and Mackey's restriction formula,   \[\langle  \wtc, \ind_{\s_A}^{\gol}(\tilde{\sigma})\rangle \neq 0\,\,  \mathrm{if\, and \, only \, if}\,\, \langle (\chi, {\chi'}, \psi_t),\tilde{\sigma}^g \rangle \neq 0 \,\, \mathrm{on}\,\, (\ztol \uol) \cap \s_A^g, \] for some $g \in \ztol \uol \backslash \gol / \s_A$. Since $\zol \uol \cap \K^{\lup}$ is a subgroup of $(\ztol \uol) \cap \s_A^g$ and $\tilde{\sigma}|_{\K^{\lup}} \cong (\psi_A)^{\oplus q^{\lup - \ldown}}$, by \autoref{lem:coset-representative-form-wtc}, we only need to consider the double coset representatives $g$ of $(\ztol \uol) \backslash \gol / \s_A$ of the form $ g =\left(\begin{smallmatrix} 1 & 0 \\ \varpi^{\ell-s}z & w   \end{smallmatrix}\right)\in \Sigma$. By  \autoref{lem:coset-representative-form-wtc}, these cosets can be further reduced to the form $\left(\begin{smallmatrix} 1 & 0 \\ 0 & w \end{smallmatrix}\right)$. Hence  we have $(\ztol \uol )\cap \s_A^g= (\ztol \uol)\cap \s_A$.
 Suppose $g = \left(\begin{smallmatrix} 1 & 0 \\ 0 & w  \end{smallmatrix} \right) $ and $g'=\left(\begin{smallmatrix} 1 & 0 \\ 0 & w'   \end{smallmatrix} \right) $ with
 \begin{gather}
 \label{eqn: wtc equalities for different cosets}
      \langle (\tilde{\chi}, \psi_t), \tilde{\sigma}^g \rangle_{(\ztol \uol) \cap \s_A} \neq 0 \,\, \mathrm{and} \,\, \langle  (\tilde{\chi}, \psi_t),  \tilde{\sigma}^{g'} \rangle_{(\ztol \uol) \cap \s_A} \neq 0.
      \intertext{To prove our result, we show that}
      (\ztol \uol) g \s_A= (\ztol \uol) g' \s_A \quad\text{and}\quad \langle (\tilde{\chi}, \psi_t), \tilde{\sigma}^g \rangle_{\ztol \uol \cap \s_A} =1.
 \end{gather}
We will consider the cases where 
$\ell$  is even and odd separately.
  \vspace{.2 cm}
  
\noindent {\bf $\ell$ is even:} Let $\ell = 2m$. In this case $\tilde{\sigma}=\widetilde{\psi_A}$. Assume $X\in \ztol \uol \cap \s_A.$ Then \autoref{eqn: wtc equalities for different cosets}  gives
\begin{gather*}
     (\tilde{\chi}, \psi_t)(X)=\widetilde{\psi_A}^g(X)  \,\,\mathrm{ and } \,  \,  (\tilde{\chi}, \psi_t)(X)=\widetilde{\psi_A}^{g'}(X),
\end{gather*}
for all $X \in \ztol \uol \cap \s_A$. The above equation can also be written as  ${  (\tilde{\chi}, \psi_t)}^g(X)=  {(\tilde{\chi}, \psi_t)}^{g'}(X)$ for all $X \in \ztol \uol \cap \s_A.$
By considering $X\in \ztol \uol\cap \s_A$ from \autoref{lem: ZtU intersection SA}, the following must be satisfied.  
\begin{gather*}
    \psi_t(\varpi^{t-f(t,i)}y(w-w'))=1 \,\, \mathrm{for \, all} \,\, y \in \cO_\ell. 
\end{gather*}
Therefore $w =w' \mod (\varpi^{f(t,i)})$.
   By \autoref{lem: wtc SNS-coset equality}, we have that $\ztol \uol g \s_A=\ztol \uol g'\s_A$.  Hence the claim is proved.
 \vspace{.2cm}  
  
\noindent {\bf $\mathbf{\ell}$ is odd:}  
By the same arguments as given for $\ell$ even, there exists at most one coset representative of $ \ztol \uol \backslash \gol /\s_A$, say $g$, such that $$\langle (\chi, \chi' , \psi_t), \tilde{\sigma}^g \rangle_{\ztol \uol \cap \s_A} \neq 0.$$ It remains to prove that $\langle (\chi, \chi' , \psi_t), \tilde{\sigma}^g \rangle_{\ztol \uol \cap \s_A} \neq 0$ implies  $\langle (\chi, \chi' , \psi_t), \tilde{\sigma}^g \rangle_{\ztol \uol \cap \s_A} = 1.$ 
We now consider the cases of $t \geq \lup$ and $t < \lup$ separately. We note that $(\ztol \uol)^g = \ztol \uol.$ Hence it is enough to prove that $\tilde{\sigma}|_{(\ztol \uol) \cap \s_A}$ is multiplicity free. 
\vspace{.2 cm}

{\bf $t \geq \lup$:}  Since $(\zol \uol)  \cap \s_A \subseteq (\ztol \uol) \cap \s_A,$ we obtain our result from \autoref{prop:multiplicity-free-SA-representation}. \vspace{.2 cm} 

{\bf $t < \lup$:} For proof of this case, we shall use the construction of $\sns$ representations as given in \autoref{sec: construction-SNS}. Consider the matrix $ \tilde{A} = \left(\begin{smallmatrix}
    \lambda-\lambda'+\varpi^i d/2  & 1\\
    \varpi^{2i}\frac{d^2}{4} &    \lambda-\lambda'+\varpi^i d/2
\end{smallmatrix}\right) \in \g(\cO_\lup)$. Let $A \in \g(\cO_{\ldown})$ be such that $\tilde{A}$ is a lift of $A$.  
 By \autoref{lem: non-regular constituents-W-type of the matrix A-wtc}(2), any $\sns$ sub-representation of $\wtc$ lies above $\psi_A$. 
Every $\sns$ representation $\rho$ of $\gol$ is of the form $\ind_{\s_A}^\gol \tilde{\sigma}$, where $\tilde{\sigma} \cong \ind_{\mathrm{N} \tC_{\gol}(\tilde{A})}^{\s_A}  (\psi_A{''})$ with $\psi_A{''}$ an extension of the character $\psi_A'$ as given in \autoref{sec: construction-SNS}. We denote  $ \tC_{\gol}(\tilde{A})$ by $C_A$ and $(\ztol \uol) \cap \s_A$ by $ \hta$. We now proceed to prove that $\mathrm{Res}^{\s_A}_{\hta}\tilde{\sigma}$ is an irreducible representation.  We have 
\[
\hta = \left\{ \begin{pmatrix} a & \varpi^{\ldown-j} b \\ 0 & a+ \varpi^{\ldown}d  \end{pmatrix}  \mid b, d \in \cO_\ell, a \in \cO_\ell^\times \right\}, 
\, \hta \cap \mathrm{N} C_A =  \left\{ \begin{pmatrix} a & \varpi^{\lup-j} b \\ 0 & a+ \varpi^{\ldown}d  \end{pmatrix}  \mid b, d \in \cO_\ell, a \in \cO_\ell^\times\right\}
\]
and $(\hta) \mathrm{N} C_A = \s_A$.  By Mackey's restriction formula,
\[
\mathrm{Res}_{\hta}^{\s_A} \ind_{\mathrm{N} C_A}^{\s_A}\psi_A{''} \cong  \ind_{(\mathrm{N}C_A) \cap \hta}^{\hta} \mathrm{Res}_{(\mathrm{N}C_A) \cap \hta}^{(\mathrm{N}C_A)}(\psi_A{''}).
\]
The result now follows via Clifford theory by observing that $\hta \cap (\mathrm{N}C_A)$ is a normal subgroup of $\hta$ and the stabilizer of $\psi_A{''}|_{\hta \cap (\mathrm{N}C_A)}$ in $\hta$ is $\hta \cap (\mathrm{N}C_A)$. This completes the proof for this case.
\end{proof}

\section{Decomposition of degenerate whittaker models of \texorpdfstring{$\GL_2(\co_\ell)$}{lg} for \texorpdfstring{$2 \leq \ell \leq 4$}{lg}}
\label{sec:examples}

In this section, we describe the decomposition of $\vtc$ into its irreducible constituents. For $t = \ell$, it was obtained by \cite{MR4356279}. We now consider $t \leq \ell-1$.   By \autoref{lem: non-regular constituents}, \autoref{cor:cuspidal-multiplicity-zero=vtc}, 
 \autoref{prop: ss in indzu-G} and 
 \autoref{cor: SNS in induction Z} and their proofs, the only part remaining towards the decomposition of $\vtc$ is to obtain the description of its $\sns$ constituents. We shall now proceed to obtain that. We will use \autoref{conditions with Kl2} in this section wherever needed. We will assume that $\ell \in \{2,3,4\}$  and $t \leq \ell-1$. 

Recall from \autoref{type of the matrix A},  $  \langle \ind_{\zol\uol}^\gol \chi\otimes\psi_t,\ind_{\s_A}^\gol \widetilde{\sigma_A} \rangle \neq 0$ for a $\sns$ matrix $A$ implies $A$ is conjugate to $\left(\begin{smallmatrix}
    \lambda/2 & b \\
    \varpi^{\ell-t} &  \lambda/2 
    \end{smallmatrix}\right)\in \M_2(\co_{\ldown})$.
\subsection{$\GL_2(\co_\ell)\, \mathrm{for}\, \ell=2$} 
For this case, since $t \leq \ell-1$, we shall assume that $A = \mat{\lambda/2}{1}{0}{\lambda/2 } \in \M_2(\cO_1).$ For $w\in \co_1,$ define a character  $\widetilde{\psi_A}^w$  of $\s_A$, such that $\widetilde{\psi_A}^w$ extends $\psi_A$, $\chi$  and satisfies $\widetilde{\psi_A}^w\mat{1}{x}{0}{1}=\psi(\varpi xw)$.  Let $\rho_w \in \mathrm{Irr}^{\mathrm{sns}}(\gol)$ be such that $\rho_w \cong \ind_{\s_A}^\gol(\widetilde{\psi_A}^w).$ We note that $\rho_w\ncong \rho_{w'} $ for  $w\neq w'$.
 \begin{lemma}
 \label{lem: o2 SNS complete decomposition}
 \begin{enumerate}

 \item For $w\in \co_1^\times$, we have $\langle V_\chi^1, \rho_w \rangle=1$. 
 \item   $\langle V_\chi^0, \rho_0 \rangle = q-1.$

 \end{enumerate} 

\end{lemma}
\begin{proof}
    Using Mackey's restriction formula, we get
    \[
    \langle \ind_{\zol\uol}^\gol 
    (\chi\otimes \psi_t) ,\ind_{\s_A}^\gol(\widetilde{\psi_A})\rangle=\oplus_{g \in \{\mat{1}{0}{0}{w} \mid w \in \co_1^\times\}} {\langle \chi\otimes \psi_t,\widetilde{\psi_A}^g\rangle}_{\zol\uol\cap \s_A^g}.
    \]
    Here $\zol\uol\cap  \s_A^g=\zol\uol $ for all $g\in \zol \uol \backslash \gol / \s_A$. 
    
    \noindent {\bf For $t=1$:} Let $g_{w}=\mat{1}{0}{0}{w^{-1}}$ for a fixed $w \in \co_1^\times$. By using the fact that $(\chi \otimes \psi_1)(X)=(\widetilde{\psi_A}^{w})^{g_{w}}(X)$  for all $X \in \zol \uol$ and the definition of $\rho_{w}$, we obtain $\langle V_\chi^1, \rho_{w} \rangle=1.$
    
    \noindent {\bf For $t=0$:} The definition of $\widetilde{\psi_A}$ gives us $\chi(X)=\widetilde{\psi_A}^g(X)$ for all $g \in \zol\uol \cap \s_A$ and for all $X \in \zol\uol$. Hence the claim.
\end{proof}

\subsection{$\GL_2(\co_\ell)\, \mathrm{for}\, \ell=3$}

For $0\leq k\leq \ell$ and $\alpha \in \cO_3$, we define a character $\psi_k^{\alpha}$ of $\uol$ as
\[
\psi_k^{\alpha}\begin{pmatrix}
    1 & x\\
    0 & 1
\end{pmatrix}=\psi_k({\alpha} x).\]

\begin{lemma}
\label{lem:sigma res to u final}
    Suppose $A=\left(\begin{smallmatrix}
          0 & 1\\ \varpi^j\beta & 0
      \end{smallmatrix}\right)\in \M_2(\co_3)$ and let $\tilde{\sigma}$ be as defined in \autoref{sec: construction-SNS}. Let $0\leq t \leq 1$ and assume $\uol\subseteq \s_A$ and $\tilde{\sigma}$ lies above $\psi_t^\alpha$ for some $\alpha \in \co_3.$ Then 
      \[\tilde{\sigma}|_{\uol}=
          \left(\underset{\varpi^2 z^2\in\co_3}{\oplus}\psi_\ell^{\varpi^2(\beta-z^2)}\right)\otimes \psi_t^\alpha.
      \] 
      
\end{lemma}
\begin{proof}
    From the construction of $\sns$ representations in \autoref{sec: construction-SNS}, we get the following
    \begin{equation}\label{eqn:resrtiction of sigma o3}
  \tilde{\sigma} |_{\uol}=  \mathrm{Res}_{\uol}^{\s_A}\ind_{\mathrm{N} \C_{\gol}(A)}^{\s_A} \psi_A^{''}=\oplus_{g\in \uol\backslash \s_A/ \mathrm{N}\C_{\gol}(A)} \ind_{(\mathrm{N} \C_{\gol}(A))^g\cap \uol}^{\uol}\mathrm{Res}_{(\mathrm{N} \C_{\gol}(A))^g\cap \uol}^{(\mathrm{N}\C_{\gol}(A))^g} {\psi_A^{''}}^g.
\end{equation}
Note that $\uol\backslash \s_A/ \mathrm{N} \C_{\gol}(A)=\{\left(\begin{smallmatrix}
    1 & 0\\
    \varpi z & 1
\end{smallmatrix}\right), z\in \co_1\}$. By the given conditions we have $(\mathrm{N} \C_{\gol}(A))^g\cap \uol=\uol$ for all $g\in \uol\backslash \s_A/ \mathrm{N} \C_{\gol}(A)$. Then \autoref{eqn:resrtiction of sigma o3} becomes
\begin{equation}
\label{eqn:sigma res u final}
      \tilde{\sigma} |_{\uol}=\oplus_{g\in \{\left(\begin{smallmatrix}
    1 & 0\\
    \varpi z & 1
\end{smallmatrix}\right)_{z\in\co_1}\}}({\psi_A^{''}}^g)|_{\uol}.
\end{equation}
Assume $X=\left(\begin{smallmatrix}
          1 & x\\ 0 & 1
      \end{smallmatrix}\right)\in \uol$. Taking $g=\left(\begin{smallmatrix}
    1 & 0\\
    \varpi z & 1
\end{smallmatrix}\right)$ and using the definition of $\psi_A^{''}$ as given in \autoref{sec: construction-SNS}, we have
      \begin{equation}
      \label{eqn: psiA'' on U calcutaion}
          ({\psi_A^{''}}^g)(X)=\psi_A^{''}(g^{-1}Xg)=\psi_A^{''}\begin{pmatrix}
        1+\varpi zx & x\\
        -\varpi^2 z^2 x & 1-\varpi zx
    \end{pmatrix}.
      \end{equation}
By the given condition $\uol \subseteq \s_A$, we must have either $(j,t)=(2,0)$ with any $\beta$ or $(j,t)=(1,1)$ with $\beta=0.$ 

\noindent\textbf{(j=2)} In this case, $g^{-1}Xg\in N$ hence \autoref{eqn: psiA'' on U calcutaion} becomes
\[
({\psi_A^{''}}^g)(X)=\psi_\ell^{\varpi^2(\beta-z^2)}(X).
\]
\textbf{(j=1 and $\mathbf{\beta=0}$)} In this case, \autoref{eqn: psiA'' on U calcutaion} becomes
\[
({\psi_A^{''}}^g)(X)=\psi_A^{''}\begin{pmatrix}
        1+\varpi zx & -\varpi z x^2\\
        -\varpi^2z^2x & 1-\varpi zx+\varpi^2z^2x^2
    \end{pmatrix}\begin{pmatrix}
        1 & x\\
        0 & 1
    \end{pmatrix}=(\psi_\ell^{\varpi^2(-z^2)}\otimes \psi_1^\alpha)(X)=\psi_\ell^{\varpi^2(\alpha-z^2)}(X).
\]

Substituting the values of $(\psi_A^{''})^g$ in \autoref{eqn:sigma res u final}, we get the desired result.
\end{proof}

Suppose, $\tilde{\sigma}|_{\zol\uol}=\chi\otimes(\underset{\varpi^2z^2\in \co_3}{\oplus}\psi_\ell^{\varpi^2(\alpha-z^2)})$ for some $\alpha\in \co_3,$ then we denote $\tilde{\sigma}$ by $\tilde{\sigma}_\alpha$ and $\rho_\alpha:=\ind_{\s_A}^\gol \tilde{\sigma_\alpha}$. Recall that $\rho_\alpha$ is an irreducible $\sns$ representation of $\gol.$
\begin{lemma}
\label{lem:t=0 o3 sns}
The multiplicities of $\rho_\beta$ in  $\ind_{\zol\uol}^{\gol}(\chi\otimes \psi_t)$ for $t=0$ are as follows:
  \begin{center}
\begin{tabular}{||c c c c||} 
 \hline 
 $\beta$ & 0 & non-zero square & non-zero non-square \\ [0.5ex] 
 \hline
 no. of distinct $\rho_{\beta}$ & $ 1 $ & $ \frac{q-1}{2} $ & $ 0 $ \\ 
 \hline
$\langle V_\chi^0,\rho_\beta \rangle$ & $q-1$ & $2(q-1)$ & $0$ \\
 \hline

\end{tabular}
\end{center}
  
\end{lemma}
\begin{proof}

   Let $A=\left(\begin{smallmatrix}
    \lambda/2 & 1 \\
    \varpi^2 \beta & \lambda/2
    \end{smallmatrix}\right) \in \M_2(\co_3)$  
    and $\tilde{\sigma}$ be as defined above. Suppose $\langle V_\chi^0, \ind_{\s_A}^{\gol}\tilde{\sigma}\rangle\neq 0$. Using definition of $V_\chi^0$ and \autoref{lem:split-semisimple from borel induction} we get
    \[
    \langle\ind_{\zol \uol}^{\gol}(\chi),\ind_{\s_A}^{\gol}\tilde{\sigma}\rangle={\langle (q-1)\chi,\oplus_{\varpi^2z^2\in \co_3}(\chi\otimes\psi^{\varpi^2({\beta-z^2})})\rangle}_{\zol\uol}.
    \]
    Here we have also used that $\zol\uol \backslash \gol /\s_A=\{\left(\begin{smallmatrix}
    1&0\\0&w
\end{smallmatrix}\right)\,|\, w\in \co_1^\times\}.$
We now explore different choices of $\beta\in \co_3$ such that $\varpi^2(\beta-z^2)=0$ or $\beta=z^2 \mod ({\varpi}).$ For $\beta=0$, the equation $-z^2=0$ has only one solution which is $z=0$, hence the intertwiner $\langle V_\chi^0,\rho_\beta \rangle$ is $q-1$.
For $\beta=k^2$ for some $k\neq 0$, $\beta-z^2=0$ has two solutions, which are $z= \pm k$. Hence the value of the intertwiner is $2(q-1)$. The number of distinct $\sns$ constituents in each case is determined by the possible values of $\beta$.

\end{proof}

\begin{lemma}
\label{lem:t=1 o3 sns}
    For $t=1$, the multiplicities of $\sns$ constituents in $\ind_{\zol\uol}^{\gol}(\chi\otimes \psi_t)$ are as follows:
    \begin{center}
\begin{tabular}{||c c c c||} 
 \hline
 $\alpha$ & 0 & non-zero square & non-zero non-square \\ [0.5ex] 
 \hline
no. of distinct $\rho_\alpha$  & $ 1 $ & $ \frac{q-1}{2} $ & $ \frac{q-1}{2} $ \\ 
 \hline
$\langle V_\chi^1,\rho_\alpha \rangle$ & $q-1$ & $q-2$ & $q$ \\
 \hline

\end{tabular}
\end{center} 

\end{lemma} 
\begin{proof}
     Let  $A=\left(\begin{smallmatrix}
    \lambda/2 & 1 \\
    0 & \lambda/2
    \end{smallmatrix}\right) \in \M_2(\co_3)$ and $\tilde{\sigma}$ be as defined above. 
Assume $  
\langle\ind_{\zol\uol}^{\gol}(\chi\otimes \psi_t),\ind_{\s_A}^{\gol}\tilde{\sigma}\rangle\neq 0
$. By \autoref{lem:sigma res to u final} we have,

\begin{align*}
     \langle \ind_{\zol\uol}^{\gol}(\chi\otimes \psi_1),\ind_{\s_A}^{\gol}\tilde{\sigma}\rangle &=\oplus_{ j \in \co_1^\times}{\langle \chi\otimes\psi_1^j,\oplus_{\varpi^2z^2\in \co_3}(\chi\otimes\psi_\ell^{\varpi^2(\alpha-z^2)})\rangle}_{\zol\uol} .   
\end{align*}
The above equality also uses the fact that $\psi_1^g\neq \psi_1^{g'}$ for any $g,g'\in \zol\uol\backslash {\gol}/\s_A$ such that $\zol\uol g \s_A\neq \zol\uol g'\s_A.$

To compute the above intertwiner, we wish to solve the equation $\varpi^2j=(\alpha-z^2)$ for $z\in \co_3$ or $j=\varpi^2(\alpha-z^2) \mod (\varpi).$ We proceed by considering various possibilities of $\alpha$.
\begin{description}
    \item[$\pmb{\alpha=0}$] In this case, the equation $j=-z^2$ has $2$ solutions  if and only if $j$ is a negative of a square in $\co_1^\times$. Also, $
|(\co_1^\times)^2|=\sfrac{(q-1)}{2}$ so, $\langle\ind_{\zol\uol}^{\gol}(\chi\otimes \psi_1),\ind_{\s_A}^{\gol}\tilde{\sigma}_\alpha\rangle=\sfrac{2(q-1)}{2}=q-1 $.

    \item[$\pmb{\alpha=k^2,\ k\neq 0}$]  For $z\in \co_1$, the equation $j=k^2-z^2$  has 2 non-zero solutions if $k^2-j$ is a square with $j \neq k^2$. This gives total $2(\sfrac{(q-1)}{2}-1)=q-3$ solutions. Also $z=0$ is always a solution. Hence the number of solutions for $z\in \co_1$ is  $1+(q-3)=q-2$. 
    
    \item[$\pmb{\alpha \neq k^2}$ for any $\pmb{k}$] In this case there are $2(q-1)/2$ solutions to the equation $j=\alpha-z^2$ for $z\in\co_1^\times$ and again for $z=0$, $\alpha=j$  is always a solution. Hence the value of the intertwiner is $1+2(q-1)/2=q$.
\end{description}
The number of such representations $\ind_{\s_A}^{\gol}\tilde{\sigma}$ is just the number of possible $\alpha\in \co_1$ in each of the above cases.
\end{proof}

\begin{lemma}
\label{lem:t=2 o3 sns}
   For $t=2$, suppose $A=\left(\begin{smallmatrix}
    \lambda/2 & 1 \\
    0 & \lambda/2
    \end{smallmatrix}\right) \in \M_2(\co_1)$. Let $\rho\in \mathrm{Irr}^\sns(\gol)$ such that $\rho=\ind_{\s_A}^\gol \tilde{\sigma}$.
    Suppose $\tilde{\sigma}|_{U(\varpi \co_\ell)}$ is non-trivial $q$-dimensional representation. Then $\langle V_\chi^2,\rho\rangle=1.$ The number of such distinct $\sns$ representations in $V_\chi^2$ is $q(q-1).$
\end{lemma}
\begin{proof}
By \autoref{cor: SNS in induction Z}, it is enough to show that 
$\langle V_\chi^2,\rho\rangle\neq 0.$ 
\[{\langle V_\chi^2, \rho\rangle}_{\zol\uol}=\oplus_{g\in\{\left(\begin{smallmatrix}
    1&0\\0&w
\end{smallmatrix}\right),{w\in\co_1^\times}\}}\langle\chi\otimes\psi_2, \tilde{\sigma}^g\rangle_{\zol\uol}
\]
We note that $\chi \otimes \psi_2$ is non-trivial on $\U(\varpi \cO_\ell).$ By the the construction of $\tilde{\sigma}$ and the given hypothesis, there exists $g = \begin{psmallmatrix}
    1 & 0 \\ 0 & w
\end{psmallmatrix} \in \gol$ such that $\langle \chi \otimes \psi_2, \tilde{\sigma}^g \rangle_{\zol \U(\varpi \cO_\ell)} \neq 0.$ By \autoref{prop:multiplicity-free-SA-representation}, $\tilde{\sigma}^g|_{\zol \uol}$ is a multiplicity free representation of dimension $q$.  Since $|\frac{\zol \uol}{\zol \U(\varpi \cO_\ell)}| = q, $ we must have $\langle \chi \otimes \psi_2, \tilde{\sigma}^g   \rangle_{\zol \uol} \neq 0.$
Therefore $\langle V_\chi^2,\rho\rangle\neq 0.$ 
\end{proof}

\subsection{$\GL_2(\co_\ell)\, \mathrm{for}\, \ell=4$}
\begin{lemma}
\label{lem:t=3 o4 sns}
    Let $A=\mat{\lambda/2}{b}{\varpi}{\lambda/2}$. Let $\tilde{\psi_A}^w$ be an extension of $\psi_A$ from $\K^2$ to $\s_A$ satisfying $\tilde{\psi_A}^w \mat{1}{\varpi x}{0}{1}=\psi(\varpi^2 x +\varpi^3 w x)$ for some $w\in \co_1$ and $\tilde{\psi_A}^w |_{\zol}=\chi |_{\zol}$. The following hold for $\rho^w \cong \ind_{\s_A}^{\gol}\tilde{\psi_A}^w$. 
    \begin{enumerate}
        \item  $\langle V_\chi^3, \rho^w\rangle=1$.
        \item The total number of $\rho^w$ such that $\rho^w \ncong \rho^{w'}$ and $\langle V_{\chi}^3,\rho^w\rangle=1$ is $q^2(q-1)$.
    \end{enumerate}

\end{lemma}
\begin{proof}
    By Mackey's restriction formula, 
    \[
    \langle V_\chi^3, \rho \rangle= \oplus _{g\in \{\mat{1}{0}{0}{w},w\in \co_2^{\times}\}}   {\langle \chi \otimes \psi_3^g, \tilde{\psi_A}^w  \rangle}_{\zol \U(\varpi \co_4)}.
    \]
    By definition of $\tilde{\psi_A}^w$, $\chi \otimes \psi_3^g=\tilde{\psi_A}^w  $  on $\zol \U(\varpi \co_4)$ if and only if $g=\mat{1}{0}{0}{1+\varpi w}$ . This proves the first part.
    Now the number of non-conjugate matrices of the form $A$ is $q-1$. There are $|\co_1| \times \frac{|\s_A|}{|\zol \U(\varpi \co_4) \K^2|}=q^2$ number of $\tilde{\psi_A}^w $ such that $\langle V_\chi^3 , \ind_{\s_A}^\gol \tilde{\psi_A}^w \rangle=1$.
    Hence the number of $\rho$ such that $\langle V_\chi^3,\rho\rangle=1$ is $q^2(q-1)$.  
\end{proof}

Suppose $A=\left(\begin{smallmatrix}
    \lambda/2 & 1\\
    0 & \lambda/2
\end{smallmatrix}\right) \in \M_2(\co_2)$. 
For $\gamma\in \co_1$, let $\psi_A^\gamma$ be an extension of $\psi_A$ from $\K^2$ to $\s_A$ satisfying $\psi_A^\gamma\left(\begin{smallmatrix}
    1 & \varpi x\\
    0 & 1
\end{smallmatrix}\right)=\psi(\varpi^3\gamma x)$. Since ${\psi_A}\left(\begin{smallmatrix}
            1 & \varpi^2 x\\ 0 & 1
        \end{smallmatrix}\right)=1$, the extension of $\psi_A$ to $\U(\varpi \co_4)$ is of the form $\psi_A^\gamma$ for some $\gamma\in \co_1$. Let $\rho_\gamma=\ind_{\s_A}^{\gol}\psi_A^\gamma$. We will use this $\psi_A^\gamma$ for $\gamma \in \cO_1$ below.

\begin{lemma}
\label{lem: o4 SNS complete decomposition}
    The multiplicities of $\sns$ constituents in $\ind_{\zol\uol}^{\gol}(\chi\otimes \psi_t)$ for $t=0,1,2$ are as follows:
\begin{enumerate}
    \item For $t=2$,
    \begin{center}
\begin{tabular}{||c c c c||} 
 \hline
$ \gamma : $ & 0 & non-zero square & non-zero non-square \\ [0.5ex] 
 \hline
 
 no. of distinct $\rho_\gamma$ & $ q $ & $ \frac{q-1}{2}q $ & $ \frac{q-1}{2}q $ \\ 
 \hline
$\langle V_\chi^2, \rho_\gamma \rangle$  & $q-1$ & $q-2$ & $q$ \\
 \hline

\end{tabular}
\end{center}
 \item For $t=1$,
    \begin{center}
\begin{tabular}{||c c c c||} 
 \hline
$ \gamma : $ & 0 & non-zero square & non-zero non-square \\ [0.5ex] 
 \hline
  no. of distinct $\rho_\gamma$ & $ q-1 $ & $ \frac{q-1}{2}q $ & $ 0 $ \\ 
 \hline
$\langle V_\chi^1, \rho_\gamma\rangle$   & $q$ & $2(q-1)$ & $0$ \\
 \hline

\end{tabular}
\end{center} \item For $t=0$,
    \begin{center}
\begin{tabular}{||c c c c||} 
 \hline
$ \gamma : $ & 0 & non-zero square & non-zero non-square \\ [0.5ex] 
 \hline
  no. of distinct $\rho_\gamma$  & $ 1 $ & $ \frac{q-1}{2}q $ & $ 0 $ \\ 
 \hline
$\langle V_\chi^0, \rho_\gamma \rangle$  & $q^2-q$ & $2(q-1)$ & $0$ \\
 \hline

\end{tabular}
\end{center}
\end{enumerate}
\end{lemma}

\begin{proof}   
By using Mackey's restriction formula, 
    \begin{equation*}
        \langle\ind_{\zol\uol}^{\gol}(\chi\otimes \psi_t),\ind_{\s_A}^{\gol}{\psi_A^\gamma}\rangle=\oplus_{g\in\{\left(\begin{smallmatrix}
            1&0\\0&v
        \end{smallmatrix}\right)\}} {\langle\chi\otimes\psi_t^g,{\psi_A^\gamma}\rangle}_{\zol\uol} \oplus_{g\in\{\left(\begin{smallmatrix}
            1 & 0\\\varpi z & w
        \end{smallmatrix}\right)\}}{\langle\chi\otimes\psi_t^g,\psi_A^\gamma\rangle}_{\zol\uol^g\cap \s_A}.
    \end{equation*}
where, $v,w\in \co_2^\times$ and $z\in\co_1^\times$.
Then the first and second intertwiner respectively gives us
    \begin{gather*}
        \psi_A^\gamma\begin{pmatrix}
            1 & x\\ 0 & 1
        \end{pmatrix}=\psi(\varpi^{\ell-t} xv) \forall x ;\\
         \psi_A^\gamma\begin{pmatrix}
            1 & \varpi x\\ 0 & 1
        \end{pmatrix}=\psi((\varpi^3 z^2 + \varpi^{\ell-t+1}w)x)  \forall x .
    \end{gather*}
By definition of  $\psi_A^\gamma$ on  $\U(\varpi \cO_\ell)$ and above expressions, we obtain the following relations:
\begin{gather}
\psi((\varpi^3 \gamma - \varpi^{\ell -t+1}v)x) = 1, \,\, \mathrm{for \,\, all}\,\, x; \label{eq:o4 SNS first intertwiner} \\
\psi( (\varpi^3 \gamma - \varpi^{\ell -t+1}w - \varpi^3 z^2)x) = 1, \,\, \mathrm{for \,\, all}\,\, x \label{eq:o4 SNS second intertwiner}. 
\end{gather}
 Now, we will consider the various possibilities of $t$ and $\gamma$.

  \vspace{.4 cm}
  
  \noindent  \textbf{(t=2)}
        We first consider $\gamma = 0$. In this case, only \autoref{eq:o4 SNS second intertwiner} can be satisfied and it gives, $\psi(\varpi^3(w+z^2)x)=1$ which gives $w=-z^2 \mod(\varpi)$. 
         Now suppose $g=\left(\begin{smallmatrix}
            1 & 0\\
            \varpi z & w
      \end{smallmatrix}\right)$ and  $g'=\left(\begin{smallmatrix}
            1 & 0\\
            \varpi z & w+\varpi w'
      \end{smallmatrix}\right)$, then
        \begin{equation}
            \label{eqn: use in t=1, Gl2(o-4) coset equality}
          \begin{pmatrix}
            1 & -w'z^{-1}w^{-1}\\
            0 & 1
        \end{pmatrix} \begin{pmatrix}
            1 & 0\\
            \varpi z & w
        \end{pmatrix} =\begin{pmatrix}
            1 & 0\\
            \varpi z & w+\varpi w'
        \end{pmatrix}  \begin{pmatrix}
            1-\varpi w^{-1}w' & -w'z^{-1}\\
            0 & 1-\varpi w^{-1}w'
        \end{pmatrix} 
        \end{equation}
        gives $\zol\uol g \s_A =\zol\uol g' \s_A$. 
        Hence $w=-z^2 \mod (\varpi)$ gives a unique $w\in \co_2^\times$ for each $z \neq 0$  such that $g=\left(\begin{smallmatrix}
            1 & 0\\ \varpi z & w
        \end{smallmatrix}\right)$ satisfies \autoref{eq:o4 SNS second intertwiner}. Hence the multiplicity of $\rho_\gamma$  in this case is $q-1$. We also note that the number of such extensions from $\uol(\varpi \co_4) $ to $\uol$ is $q.$
   Now onwards, in view of \autoref{eqn: use in t=1, Gl2(o-4) coset equality}, we will assume that $w$ is determined modulo $\varpi$.
      \vspace{1pt}
     Next we assume that $\gamma \in {(\cO_1^\times)}^2.$ In this case \autoref{eq:o4 SNS first intertwiner}  is satisfied for $v=\gamma$ and \autoref{eq:o4 SNS second intertwiner} is satisfied for all $z \in \cO_1^\times$ such that $z^2\neq \gamma$ and therefore $w=\gamma-z^2$. Total number of such $g\in \zol\uol \backslash {\gol}/\s_A$ is $1+(q-3)=q-2$. We further note that the number of such $\rho_\gamma$ giving non-zero intertwiner is $\frac{(q-1)q}{2} $.
     
      \vspace{1pt}
     Finally, we assume that $\gamma \in \cO_1^\times \setminus (\cO_1^\times)^2$ is a nonzero non-square. Then again \autoref{eq:o4 SNS first intertwiner} is satisfied for $v=\gamma$ and \autoref{eq:o4 SNS second intertwiner} is satisfied for all $z \in \cO_1^\times$ which are non squares and $w=\gamma-z^2$. Total number of such $g\in \zol\uol \backslash {\gol}/\s_A$ is $q$. Hence multiplicity of $\rho_\gamma$ in $\vtc$ for this case is $q$. Further, the number of distinct $\rho_\gamma$  is $\frac{(q-1)q}{2} $.

\vspace{.4 cm} 
 \noindent \textbf{(t=1)}  We first consider $\gamma=0$. In this case, \autoref{eq:o4 SNS second intertwiner} can not be satisfied for any $z$ and $w$ whereas \autoref{eq:o4 SNS first intertwiner} will be satisfied for any $v\in\co_2^\times$. Further, for a fixed extension $\psi_A^\gamma$ of $\uol$ and  $g=\left(\begin{smallmatrix}
            1 & 0 \\ 0 & v
        \end{smallmatrix}\right)$,  $g'=\left(\begin{smallmatrix}
            1 & 0\\ 0 & v'
        \end{smallmatrix}\right)$, we have $\psi_t^g = \psi_t^{g'}$ if and only if  
        $v-v' \in (\varpi)$. Hence the total number of distinct cosets, contributing to the multiplicity of $\rho_\gamma$  in $V_\chi^1$, is $q$. Also the number of such distinct $\ind_{\s_A}^{\gol}\psi_A^\gamma$ is $q-1$.
        \vspace{1pt}  
        
    We next consider the case of non-zero square  $\gamma$. For this case only \autoref{eq:o4 SNS second intertwiner} is satisfied with  $z \in \{ \gamma,-\gamma\}$ and  $w\in \co_1^\times$ using \autoref{eqn: use in t=1, Gl2(o-4) coset equality}. Hence the multiplicity of $\rho_\gamma$ in $V_\chi^1$ is $2(q-1)$. The number of $\rho_\gamma$ in $V_\chi^1$ equals the number of squares in $\co_1^\times $ times the number of extensions from $\uol(\varpi \co_4)$ to $\uol$. This gives the total number to be $\frac{(q-1)q}{2} $. 
    
    \vspace{1pt}
   Finally, for non-zero non-square $\gamma$, $\psi_A^\gamma$ does not satisfy any of the two equations. Hence no contribution is obtained from this case. 

\vspace{.4 cm} 

\noindent \textbf{(t=0)} 
We first consider $\gamma=0$.
      In this case, \autoref{eq:o4 SNS first intertwiner} is satisfied for every $v\in \co_2^\times$ and \autoref{eq:o4 SNS second intertwiner} is not satisfied. Hence multiplicity in this case is $q^2-q$. Further, $\psi_A^\gamma$ is determined by $\psi_t$ here. 

      \vspace{1pt}
  The cases of nonzero $\gamma$ can be handled in a similar way as $t=1$. We omit the details here. 
\end{proof}

  By \autoref{coro:number-ss-sns-vtc}, the number of $\sns$ constituents (counted with multiplcity) of $\vtc$ is $q^{\ell -2}(q-1).$ By counting the multiplicities from ~\cref{lem: o2 SNS complete decomposition,lem:t=0 o3 sns,lem:t=1 o3 sns,lem:t=2 o3 sns,lem:t=3 o4 sns,lem: o4 SNS complete decomposition},  we get the required number $q^{\ell-2}(q-1)$. Therefore,  we obtain all $\sns$ constituents of $\vtc$.

\begin{proof}[Proof of \autoref{thm:structure-endomorphism-algebra}]
This is easily obtained from the above discussion regarding the decomposition of $\vtc$ into its irreducible constituents.     
\end{proof}

\subsection*{Acknowledgments}
 The authors sincerely thank Uri Onn and Santosh Nadimpalli for their valuable discussions and insightful feedback on this work. The second named author gratefully acknowledges the financial support provided by SERB, India, through grant SPG/2022/001099.

\bibliography{refs}
\end{document}